\documentclass[hidelinks, 10pt]{article}

\usepackage[UKenglish]{babel}
\usepackage[utf8]{inputenc}
\usepackage{amsmath}
\usepackage{amsthm}
\usepackage{amsfonts}
\usepackage{amssymb}
\usepackage{mathtools}
\usepackage{dsfont} 
\usepackage[outline]{contour}

\usepackage[dvipsnames]{xcolor}
\usepackage{hyperref}
\usepackage{lmodern}
\usepackage{fullpage}
\usepackage{microtype}
\usepackage{newpxmath,newpxtext}
\usepackage{booktabs}
\usepackage{enumitem} 
\usepackage{comment}

\usepackage{tikz}
\usetikzlibrary{positioning}
\usetikzlibrary{arrows,backgrounds}
\usetikzlibrary{decorations.pathmorphing}
\usepackage{tikz-cd}
\usepackage{adjustbox}

\tikzset{node distance=2cm, auto}
\tikzcdset{row sep/normal=2.7em,column sep/normal=3.5em}

\allowdisplaybreaks

\setlist[description]{font=\normalfont\bfseries\space}

\renewcommand\labelenumi{(\roman{enumi})}
\renewcommand\theenumi\labelenumi

\hypersetup{
    colorlinks=true,
    linkcolor=gray,
    filecolor=gray,      
    urlcolor=gray,
    citecolor=gray
    }
\numberwithin{equation}{section}

\newtheorem{definition}{Definition}[section]
\newtheorem{proposition}[definition]{Proposition}
\newtheorem{lemma}[definition]{Lemma}
\newtheorem{theorem}[definition]{Theorem}
\newtheorem{corollary}[definition]{Corollary}

\theoremstyle{definition}
\newtheorem{example}[definition]{Example}

\newtheorem{remark}[definition]{Remark}

\newcommand{\CategoryFont}[1]{\mathsf{#1}}

\renewcommand{\bar}[1]{\mkern 1.5mu\overline{\mkern-1.5mu#1\mkern-1.5mu}\mkern 1.5mu}

\newcommand{\Famly}{\mathscr{F}}
\newcommand{\Dsply}{\mathscr{D}}
\newcommand{\Smer}{\mathscr{S}}

\newcommand{\Etal}{\acute{\mathscr{E}}}
\newcommand{\Open}{\CategoryFont{Open}}

\renewcommand{\,}{,\dots,}

\newcommand{\1}{\mathds{1}}

\renewcommand{\=}{\colon\kern-1ex=}
\renewcommand{\epsilon}{\varepsilon}
\renewcommand{\o}{\circ}
\renewcommand{\b}{\bullet}
\newcommand{\<}{\langle}
\renewcommand{\>}{\rangle}
\newcommand{\op}{\mathsf{op}}

\renewcommand{\^}[1]{^{(#1)}}
\renewcommand{\*}{\ast}
\newcommand{\nto}{\nrightarrow}
\newcommand{\Nat}{\mathsf{Nat}}

\newcommand{\N}{\mathbb{N}}

\newcommand{\R}{\mathbb R}

\newcommand{\id}{\mathsf{id}}

\newcommand{\End}{\CategoryFont{End}}

\newcommand{\X}{\mathbb{X}}

\newcommand{\cRing}{\CategoryFont{cRing}}

\newcommand{\MCat}{\M\text{-}\CategoryFont{Cat}}

\newcommand{\MTngCat}{\M\text{-}\CategoryFont{TngCat}}

\newcommand{\cTngCat}{\CategoryFont{cTngCat}}

\newcommand{\Smooth}{\CategoryFont{Smooth}}

\newcommand{\cAlg}{\CategoryFont{cAlg}}
\newcommand{\Slice}{\CategoryFont{Slice}}
\newcommand{\Split}{\CategoryFont{Split}}
\newcommand{\TngPair}{\CategoryFont{TngPair}}
\newcommand{\Term}{\CategoryFont{Term}}

\newcommand{\M}{\mathscr{M}}
\newcommand{\sRestrCat}{\CategoryFont{sRestrCat}}
\newcommand{\Par}{\CategoryFont{Par}}

\newcommand{\Weil}{\CategoryFont{Weil}_1}

\renewcommand{\d}{\mathsf{d}}
\newcommand{\T}{\mathrm{T}}

\newcommand{\TT}{\mathbb{T}}

\newcommand{\DBnd}{\CategoryFont{DBnd}}

\newcommand{\U}{\mathrm{U}}

\title{Pullbacks in tangent categories and\\
tangent display maps}
\author{Geoffrey Cruttwell and Marcello Lanfranchi}
\date{}
\makeatother
\makeatletter

\begin{document}

\maketitle

\begin{abstract}\noindent
In differential geometry, the existence of pullbacks is a delicate matter, since the category of smooth manifolds does not admit all of them. When pullbacks are required, often submersions are employed as an ideal class of maps which behaves well under this operation and the tangent bundle functor. This issue is reflected in tangent category theory, which aims to axiomatize the tangent bundle functor of differential geometry categorically. Key constructions such as connections, tangent fibrations, or reverse tangent categories require one to work with pullbacks preserved by the tangent bundle functor. In previous work, this issue has been left as a technicality and solved by introducing extra structure to carry around. This paper gives an alternative to this by focusing on a special class of maps in a tangent category called tangent display maps; such maps are well-behaved with respect to pullbacks and applications of the tangent functor. We develop some of the general theory of such maps, show how using them can simplify previous work in tangent categories, and show that in the tangent category of smooth manifolds, they are the same as the submersions. Finally, we consider a subclass of tangent display maps to define open subobjects in any tangent category, allowing one to build a canonical split restriction tangent category in which the original one naturally embeds.  

\end{abstract}
\noindent

\tableofcontents

\section{Introduction}
\label{section:introduction}
Tangent categories axiomatize one of the central features of differential geometry: the tangent bundle functor on the category of smooth manifolds. However, one of the recurring issues when working with tangent categories is the existence of pullbacks and their preservation by the tangent bundle functor and its iterates. In particular, in the standard setting for differential geometry (the category of smooth manifolds), not all pullbacks exist, and those that do not need be ``correct'', in that they need not be preserved by the tangent bundle functor. This paper gives one way to approach this issue, by considering the \emph{tangent display} maps in a tangent category.

To help understand the problems in greater depth, let us begin by looking at some of the issues with pullbacks in the category of smooth manifolds. For one example, the pullback of the following diagram
\[
\begin{tikzcd}
& \R^2 \\
1 & \R
\arrow["0"', from=2-1, to=2-2]
\arrow["m", from=1-2, to=2-2]
\end{tikzcd}
\]
where $m(x,y) = xy$, and $0$ picks out the point $0 \in \R$, is the union of the two co-ordinate axes in $\R^2$, which cannot be given the structure of a smooth manifold, and hence the pullback does not exist in the category of smooth manifolds.

On the other hand, the pullback 
\[
\begin{tikzcd}
& \R \\
1 & \R
\arrow["0"', from=2-1, to=2-2]
\arrow["s", from=1-2, to=2-2]
\end{tikzcd}
\]
where $s(x) = x^2$, \emph{does} exist in the category of smooth manifolds: it is simply a single point. However, when we apply the tangent bundle functor $\T$ to this pullback, we get the diagram
\[
\begin{tikzcd}
1 & \T(\R) \\
1 & \T(\R)
\arrow["{(0,0)}", from=1-1, to=1-2]
\arrow[from=1-1, to=2-1]
\arrow["{(0,0)}"', from=2-1, to=2-2]
\arrow["\T(s)", from=1-2, to=2-2]
\end{tikzcd}
\]
where $\T(s)(x,v) = (x^2,2xv)$. However, this is not a pullback diagram - the pullback of the right and bottom arrows is $\R$, since for any $v \in \R$, $\T(s)(0,v) = (0,0)$. 

Thus, since tangent categories aim to give a common framework for any setting in which one can discuss differentials and tangents (including, in particular, the category of smooth manifolds), in an arbitrary tangent category, one cannot assume all pullbacks exist, or that they are preserved by the tangent bundle functor.

Nevertheless, if one wants to recreate various constructions of differential geometry in an arbitrary tangent category, one is naturally forced to consider various pullbacks. Here are some examples of this:
\begin{itemize}
\item Already in the definition of a tangent category \cite[Definition 2.1]{cockett:differential-bundles} two kinds of pullbacks appear: (i) for each $n \in \N$ and object $M$, the pullback of $n$ copies of the projection $p_M: \T M \to M$, $\T_nM$, is required to exist; $\T_2M$ in particular is needed as the domain for addition, and (ii)  the ``universality of vertical lift'' axiom\footnote{Note that in the original paper, this was given as an equalizer.} asks for a certain diagram to be a pullback which is preserved by $\T$. 
\item Similar pullbacks to the previous point are needed to define differential bundles \cite[Definition 2.1]{cockett:differential-bundles}.
\item To define a connection on a differential bundle $q: E \to M$, one needs to consider the associated ``horizontal bundle'', which is given as the pullback of $p: \T M \to M$ along $q$ \cite[4.1]{cockett:connections}. As we shall discuss in Section \ref{section:conclusions}, similar issues exist when one wants to define connections on more general maps.
\item To build a tangent category structure on the slice of a tangent category, one needs various pullbacks to exist (and be preserved by $\T$) \cite[Prop. 5.7]{cockett:differential-bundles}.
\item To build a fibration of differential bundles, one needs the pullback of a differential bundle along any map to exist (and be preserved by $\T$) \cite[Prop. 5.7]{cockett:differential-bundles}.
\item To discuss ``reverse tangent'' structure, it is useful to have a fibration of differential bundles, which, as above, requires pullbacks of differential bundles along arbitrary maps to exist \cite[Defn. 24]{cruttwell:reverse-tangent-cats}.
\end{itemize}

Thus, a natural question is to determine how to handle the existence of these pullbacks in an arbitrary tangent category. Here are some ways this issue has been dealt with in previous work, and the drawbacks they bring:
\begin{itemize}
\item The most common way to handle the existence of such pullbacks is to simply ask for them as needed: for example, this was how the issue was handled in \cite{cockett:tangent-cats}. However, this gives the theory an ad-hoc feel, and can make it unclear exactly which pullbacks are needed in general. 

\item As a result, some works have asked for a specified system of pullbacks and/or maps along which all pullbacks exist; such an approach was first introduced in \cite{cockett:differential-bundles} (here we call such a system a \emph{tangent display system}). However, this is an additional structure that then needs to be ``carried around'' as one passes to new tangent categories constructed from existing ones (such as tangent categories of vector fields or tangent categories of connections). This viewpoint also does not seem to match examples very well: it is very rare in differential-geometry-like settings that one needs to specify some system of pullbacks. 

\item In the category of smooth manifolds, the important maps along which all pullbacks exist are the submersions. Thus, any questions about pullbacks in this particular tangent category are usually only considered along such maps. One can give an abstract definition of a submersion in an arbitrary tangent category (see Definition \ref{definition:submersions}); however, there is no reason in a general tangent category why pullbacks along such maps should exist or be preserved by the tangent bundle functor. 

\item Any tangent category embeds in a tangent category in which all pullbacks exist and are preserved by $\T$ (see \cite{garner:embedding-theorem-tangent-cats}). Thus, in theory, any time one needs an appropriate pullback, one could ask for it ``in a larger tangent category in which this exists''. However, the existence of such pullbacks is not just needed to make proofs easier - as noted above, it is needed in many of the fundamental definitions of the subject itself. Given this, it is quite awkward to have definitions themselves rely on the existence of some larger tangent category which may or may not be of interest. 
\end{itemize}
In this paper, the solution we propose to this problem is to focus on well-behaved maps which we call ``tangent display maps''. These are the maps in a tangent category which enjoy all the properties one could hope for in relation to pullbacks:
\begin{itemize}
\item All pullbacks along them exist.
\item All such pullbacks are preserved by all powers of the tangent bundle functor.
\item Applying any power of the tangent bundle functor to such a map produces another such map.
\item The pullback of such a map along any morphism produces another such map.
\end{itemize}
We show that the collection of such maps forms the \emph{maximal} tangent display system. We also show a number of other important results related to such maps:
\begin{itemize}
\item (Theorem \ref{theorem:classification-submersions}) In the tangent category of smooth manifolds, such maps are \emph{precisely} the submersions.
\item (Theorem \ref{theorem:fully-display-tangent-cats}) A natural condition is to ask that the differential bundles in a tangent category are tangent display maps (as is the case in the category of smooth manifolds, where vector bundle projections are submersions). 
We show that under relatively mild conditions, this is true so long as each tangent bundle is itself a tangent display map. 
\item (Theorem \ref{theorem:adjunction-Term-Slice}) We show that a subcategory of the slice of a tangent category (focusing on the tangent display maps) is well-defined and enjoys a strong universal property.
\item (Sections \ref{subsection:reverse-tangent-category} and \ref{subsection:linear-completeness}) Several constructions and/or results in tangent category papers are greatly simplified by focusing one's assumptions on tangent display maps (as opposed to using some of the other solutions mentioned above). 
\item (Section \ref{subsection:restriction}) By focusing on certain tangent display maps, one can generalize the notion of an open subobject from differential geometry; these subobjects then provide a natural choice of split (tangent) restriction category in which the original tangent category naturally embeds.     
\end{itemize}

Thus, given these advantages, we argue that going forward, whenever one needs the existence of ``nice'' pullbacks in a tangent category, one should ask that the appropriate maps be tangent display maps. In other words, following from the principle that the best way to work is at the \emph{right} level of generality, not the maximum level of generality, we believe this is exactly the right level of generality for working with pullbacks in tangent categories, and that our results about these maps in this paper reflect this.

\subsection{Notation}
\label{subsection:notation}
We denote the composite of morphisms $f\colon A\to B$ and $g\colon B\to C$ of a category in the diagrammatic order, i.e., $fg\colon A\to C$. For functors evaluated on objects or morphisms, we adopt the applicative notation, i.e., $FGA$ or $FGf$.  We use $(\X,\TT)$ to denote a tangent category. 

\subsection{Acknowledgements}
\label{subsection:acknowledgments}
For this paper, the corresponding author for this paper is Geoff Cruttwell (gcruttwell@mta.ca).  There is no external data associated with this paper.  Both authors worked equally on the research and writing aspects of the paper.  Neither author has any competing interests.

\section{Display maps and tangent display maps}
The goal of this section is to introduce the protagonist of this paper: tangent display maps. We start by discussing the concept of a tangent display system, which was one of the solutions suggested in previous work to the problem of having a class of maps which behaves well under pullbacks and the tangent bundle functor. We then introduce tangent display maps and show that they form the maximal tangent display system.
\par Following this, we then discusses sufficient conditions for which tangent display maps are closed under retract. This plays a role in the applications we explore later in the paper.  One of these sufficient conditions is linked to Cauchy completion. We briefly discuss this relationship and show that the Cauchy completion of a tangent category is still a tangent category which preserves the tangent display maps.  Finally, at the end of this section we prove the promised equivalence between tangent display maps and submersions in the category of smooth manifolds.  
\par We assume the reader is familiar with the theory of tangent categories, as presented in~\cite{cockett:tangent-cats}.

\subsection{Tangent display systems}
\label{subsection:tangent-display-systems}
We begin by setting up definitions involving pullbacks and their preservation by powers of an endofunctor.

\begin{definition}
\label{definition:display-maps}
In any category $\X$:
\begin{itemize}
\item A morphism $q\colon E\to M$ of $\X$ is a \textbf{display map} if for any morphism $f\colon N\to M$ the pullback of $q$ along $f$ exists. The class of display maps of a category is denoted by $\Dsply(\X)$.
\item A family of morphisms $\Famly$ of $\X$ is \textbf{closed under pullbacks} if, whenever the pullback of a morphism $q\colon E\to M$ of $\Famly$ along a generic morphism $f\colon N\to M$ of $\X$ exists, then the pulled-back morphism $N\times_ME\to N$ is also a morphism of $\Famly$.
\item A family $\Dsply$ of morphisms of $\X$ is a \textbf{display system} if each morphism $q$ of $\Dsply$ is a display map and $\Dsply$ is closed under pullbacks.
\end{itemize}
\end{definition}

\begin{definition}
\label{definition:T-limits}
Let $(\X,\T)$ denote a category $\X$ equipped with an endofunctor $\T\colon\X\to\X$.
\begin{itemize}
\item A \textbf{$\T$-limit} is a limit diagram preserved by all iterates $\T^n$ of the endofunctor $\T$, for every positive integer $n$. When $(\X,\TT)$ is a tangent category and $\T$ represents the tangent bundle functor, a $\T$-limit is also called a \textbf{tangent limit}. When the $\T$-limit is of some shape $S$, we adopt the convention to call it a $\T$-``name of the limit shape'', e.g., a $\T$-pullback is a pullback diagram which is a $\T$-limit.
\item A morphism $q\colon E\to M$ of $(\X,\T)$ \textbf{admits all $\T$-pullbacks} if $q$ is a display map and for every morphism $f\colon N\to M$, the pullback of $q$ along $f$ is a $\T$-pullback.
\item A family of morphisms $\Famly$ of $(\X,\T)$ is \textbf{closed under $\T$-pullbacks} if, whenever a $\T$-pullback of a morphism $q\colon E\to M$ of $\Famly$ along a morphism $f\colon N\to M$ exists, the pulled-back morphism $N\times_ME\to N$ is also a morphism of $\Famly$.
\item A family of morphisms $\Famly$ is \textbf{stable under $\T$-pullbacks} if it is closed under $\T$-pullbacks and each morphism of $\Famly$ admits all $\T$-pullbacks. 
\item A \textbf{$\T$-display system} is a family $\Dsply$ of morphisms of $\X$ which is stable under $\T$-pullbacks and stable under $\T$. Concretely, this last condition means that whenever $q$ is in $\Dsply$ so is $\T q$. When $(\X,\TT)$ is a tangent category and $\T$ represents the tangent bundle functor, a $\T$-display system is also called a \textbf{tangent display system}.
\end{itemize}
\end{definition}

\subsection{Tangent display maps}
\label{section:tangent-display-maps}
It is a standard result in category theory that the left square of a diagram of type:
\begin{equation}
\label{equation:outer-right-left-pullback}
\begin{tikzcd}
A & B & C \\
{M'} & {B'} & {C'}
\arrow[from=1-2, to=2-2]
\arrow[from=1-2, to=1-3]
\arrow[from=2-2, to=2-3]
\arrow[from=1-3, to=2-3]
\arrow[from=2-1, to=2-2]
\arrow[from=1-1, to=1-2]
\arrow[from=1-1, to=2-1]
\end{tikzcd}
\end{equation}
is a pullback diagram provided that the right and the outer squares are. The next lemma extends this result to $\T$-pullbacks.

\begin{lemma}
\label{lemma:pullback-lemma}
Let $\X$ be a category equipped with an endofunctor $\T\colon\X\to\X$. Consider also the diagram of Equation~\eqref{equation:outer-right-left-pullback}. If the outer square and the right square are $\T$-pullback diagrams, so is the left square.
\begin{proof}
The proof is a straightforward exercise of category theory we leave to the reader to spell out.
\end{proof}
\end{lemma}

The following seems to be folklore, but we could not find a proof of it, so we also record it here, along with a proof, as it will also be useful in what follows. 

\begin{proposition}
\label{proposition:display-maps-form-maximal-display-system}
For a category $\X$, the family $\Dsply(\X)$ of display maps of $\X$ forms the (unique) maximal display system of $\X$ with respect to inclusion. Moreover, $\Dsply(\X)$ is closed under composition.
\begin{proof}
Consider first a display map $q\colon E\to M$ and let $f\colon N\to M$ be any morphism of $\X$. Since $q$ admits all pullbacks the pullback of $q$ along $f$ exists. As a shorthand, let $q'\colon E'\to N$ denote the pullback of $q$ along $f$. We want to show that $q'$ also admits all pullbacks. So, let $g\colon P\to N$ be another morphism of $\X$ and consider the following diagram:
\begin{equation*}
\begin{tikzcd}
{E''} & {E'} & E \\
P & N & M
\arrow["{q'}"{description}, from=1-2, to=2-2]
\arrow[from=1-2, to=1-3]
\arrow["f"', from=2-2, to=2-3]
\arrow["q", from=1-3, to=2-3]
\arrow["g"', from=2-1, to=2-2]
\arrow["{q''}"', from=1-1, to=2-1]
\arrow["gf"', bend right, from=2-1, to=2-3]
\arrow[bend left, from=1-1, to=1-3]
\end{tikzcd}
\end{equation*}
where $q''\colon E''\to P$ denotes the pullback of $q$ along the composition $gf$, which exists since $q$ admits all pullbacks. However, this implies the existence of a unique morphism, indicated by a dash:
\begin{equation*}
\begin{tikzcd}
{E''} & {E'} & E \\
P & N & M
\arrow["{q'}"{description}, from=1-2, to=2-2]
\arrow[from=1-2, to=1-3]
\arrow["f"', from=2-2, to=2-3]
\arrow["q", from=1-3, to=2-3]
\arrow["g"', from=2-1, to=2-2]
\arrow["{q''}"', from=1-1, to=2-1]
\arrow[dashed, from=1-1, to=1-2]
\arrow["gf"', bend right, from=2-1, to=2-3]
\arrow[bend left, from=1-1, to=1-3]
\end{tikzcd}
\end{equation*}
It is clear that the outer and the right squares are pullback diagrams, so thanks to Lemma~\ref{lemma:pullback-lemma}, so is the left square. Therefore, $q'$ is also a display map. In particular, $q'$ belongs to $\Dsply(\X)$, i.e., $\Dsply(\X)$ is closed under pullbacks, thus $\Dsply(\X)$ is a display system. The next step is to show that $\Dsply(\X)$ is the maximal display system. However, this is immediate since for a morphism to belong to a display system means being a display map.
\par Finally, to see that $\Dsply(\X)$ is closed under composition, consider two display maps $q\colon E\to M$ and $p\colon P\to E$ and let $f\colon N\to M$ be any morphism of $\X$. Since $q$ is a display map, $q$ admits a pullback along $f$; let $\pi_2\colon N\times_ME\to E$ be the projection to $E$. Since $p$ is a display map, the pullback of $p$ along $\pi_2$ is well-defined. It is not hard to see that the resulting pullback diagram is the pullback of the composition $pq$ of $p$ and $q$ along $f$.
\end{proof}
\end{proposition}

\begin{definition}
\label{definition:T-display-maps}
If $\T$ is an endofunctor on a category $\X$, a \textbf{$\T$-display map} $q\colon E\to M$ is a morphism such that for every $n\geq0$, $\T^nq$ (when $n=0$, $\T^0q=q$ and for $n=1$, $\T^1q=\T q$) admits all $\T$-pullbacks. When $(\X,\TT)$ is a tangent category and $\T$ represents the tangent bundle functor, a $\T$-display map is also called a \textbf{tangent display map}. In the following, let $\Dsply(\X,\T)$ denote the family of $\T$-display maps of a category $\X$ equipped with an endofunctor $\T$. When $(\X,\TT)$ is a tangent category and $\T$ represents the tangent bundle functor, we adopt the notation $\Dsply(\X,\TT)$ for $\Dsply(\X,\T)$.
\end{definition}

The following gives a slightly expanded version of this definition:

\begin{lemma}
\label{proposition:equivalent-definition-tangent-display-maps}
A morphism $q\colon E\to M$ of a tangent category $(\X,\TT)$ is a tangent display map if and only if for any $n\geq0$ and any morphism $f\colon N\to\T^nM$ there are morphisms $\pi_1\colon P\to N$ and $\pi_2\colon P\to E$ such that for any $m\geq0$ the diagram:
\begin{equation*}
\begin{tikzcd}
{\T^mP} & {\T^{n+m}E} \\
{\T^mN} & {\T^{n+m}M}
\arrow["{\T^m\pi_2}", from=1-1, to=1-2]
\arrow["{\T^m\pi_1}"', from=1-1, to=2-1]
\arrow["{\T^{n+m}q}", from=1-2, to=2-2]
\arrow["{\T^mf}"', from=2-1, to=2-2]
\end{tikzcd}
\end{equation*}
is a pullback. 
\end{lemma}

\begin{example}
\label{example:trivial-tangent-display-map}
Each category $\X$ comes equipped with a trivial tangent structure $\1$, whose tangent bundle functor is the identity functor and whose structural natural transformations are the identities. In such a trivial tangent category, tangent display maps coincide with display maps.
\end{example}

\begin{example}
\label{ex:display_in_AG}
In the tangent categories of affine schemes and schemes (either the general categories or over a fixed base; see~\cite[Section 4.1]{cruttwell:algebraic-geometry}), the tangent display maps are all maps, since in all cases these categories are complete and the tangent bundle functor is a right adjoint.
\end{example}

\begin{example}
In the tangent category associated to a model of SDG (that is, the full subcategory of microlinear objects; see~\cite[Prop. 5.10]{cockett:tangent-cats}), the category is cartesian closed and the tangent bundle functor is a right adjoint, so again tangent display maps are all maps. 
\end{example}

\begin{example}
We will see later (Theorem \ref{theorem:classification-submersions}) that in the tangent category of smooth manifolds, tangent display maps are the same as submersions.   
\end{example}

Our first result about these maps is:

\begin{theorem}
\label{theorem:T-display-maps-form-maximal-T-display-system}
$\Dsply(\X,\T)$ forms the (unique) maximal tangent display system of $(\X,\T)$ with respect to inclusion. Moreover, $\Dsply(\X,\T)$ is closed under composition.
\begin{proof}
It is straightforward to see that if $q$ is a $\T$-display map, so is $\T q$. Let us prove that $\Dsply(\X,\T)$ is stable under $\T$-pullbacks. Consider a $\T$-display map $q\colon E\to M$ and let $f\colon N\to M$ be a morphism of $\X$. Since $q$ is $\T$-display, the $\T$-pullback of $q$ along $f$ exists. Let $q'\colon E'\to N$ denote the pullback of $q$ along $f$. The goal is to show that $q'$ is still $\T$-display. Consider another morphism $g\colon P\to N$. Since $q$ is $\T$-display, $q$ also admits the $\T$-pullback along the composite $gf\colon P\to M$. Thus, we obtain the diagram:
\begin{equation*}
\begin{tikzcd}
{E''} & {E'} & E \\
P & N & M
\arrow[dashed, from=1-1, to=1-2]
\arrow[bend left, from=1-1, to=1-3]
\arrow["{q''}"', from=1-1, to=2-1]
\arrow[from=1-2, to=1-3]
\arrow["{q'}", from=1-2, to=2-2]
\arrow["\lrcorner"{anchor=center, pos=0.125}, draw=none, from=1-2, to=2-3]
\arrow["q", from=1-3, to=2-3]
\arrow["g"', from=2-1, to=2-2]
\arrow["f"', from=2-2, to=2-3]
\end{tikzcd}
\end{equation*}
where the dashed arrow is induced by the universality of the right square. Since the outer and the right squares are $\T$-pullbacks, by Lemma~\ref{lemma:pullback-lemma}, so is the left square.
\par The next step is to show that each $\T^nq$ also admits all $\T$-pullbacks, for each $n\geq0$. However, since the pullback of $q$ along $f$ is a $\T$-pullback, $\T^nq'$ is the pullback of $\T^nq$ along $\T^nf$. However, $\T q$ is also a $\T$-display map, and by induction, so is $\T^nq$. This proves that $\Dsply(\X,\T)$ is a $\T$-display system.
\par Finally, if $\Famly$ is a $\T$-display system and $q$ is a morphism of $\Famly$, then for every $m\geq0$, $\T^mq$ admits all $\T$-pullbacks, which is precisely the definition of a $\T$-display map. So, each $\T$-display system is a subfamily of $\Dsply(\X,\T)$. This implies that $\Dsply(\X,\T)$ is the maximal $\T$-display system with respect to inclusion. Finally, if $q\colon E\to M$ and $p\colon P\to E$ are two $\T$-display maps, by Proposition~\ref{proposition:display-maps-form-maximal-display-system}, the composite $pq$ of $p$ with $q$  is also a display map. Finally, using a similar argument, it is not hard to see that the composite $pq$ is also a $\T$-display map.
\end{proof}
\end{theorem}

\begin{definition}
\label{definition:display-tangent-categories}
A \textbf{well-displayed tangent category} is a tangent category whose tangent bundles are tangent display maps.
\end{definition}


\subsection{Retractive display systems}
\label{subsection:retracts}
In this section we consider how tangent display maps interact with retractions. This will be useful not only as a general result of interest, but also in particular in Theorem \ref{theorem:fully-display-tangent-cats}.

\begin{definition}
\label{definition:retractivity}
A family of morphisms $\Famly$ of a category $\X$ is \textbf{retractive} when it is stable under retracts. Concretely, this means that for each $f\colon P\to M$ in $\Famly$ and each section-retraction pair $s\colon E\to P$, $r\colon P\to E$, the composition $sf\colon E\to M$ is also in $\Famly$. A \textbf{retractive display system} of a category $\X$ is a display system which is also retractive. Similarly, a \textbf{retractive tangent display system} of a tangent category $(\X,\TT)$ is a tangent display system which is also retractive.
\end{definition}

\begin{definition}
\label{definition:retractive-tangent-categories}
A tangent category $(\X,\TT)$ is \textbf{retractive} if the tangent display system $\Dsply(\X,\TT)$ of tangent display maps of $(\X,\TT)$ is retractive.
\end{definition}

In the following, we denote a section-retraction pair $s\colon E\to P, r\colon P\to E$, where $sr=\id_E$, by $(s,r)\colon E\leftrightarrows P$. The next lemma was proved in~\cite[Lemma~1.2.8]{macadam:vector-bundles}. Recall that a weak $\T$-pullback is a commutative square diagram which satisfies a similar universal property of a $\T$-pullback which, however, does not require the uniqueness of the induced morphism.

\begin{lemma}
\label{lemma:pullbacks-retraction}
(Weak) $\T$-pullbacks are stable under retracts. Concretely, consider the following commutative diagram:
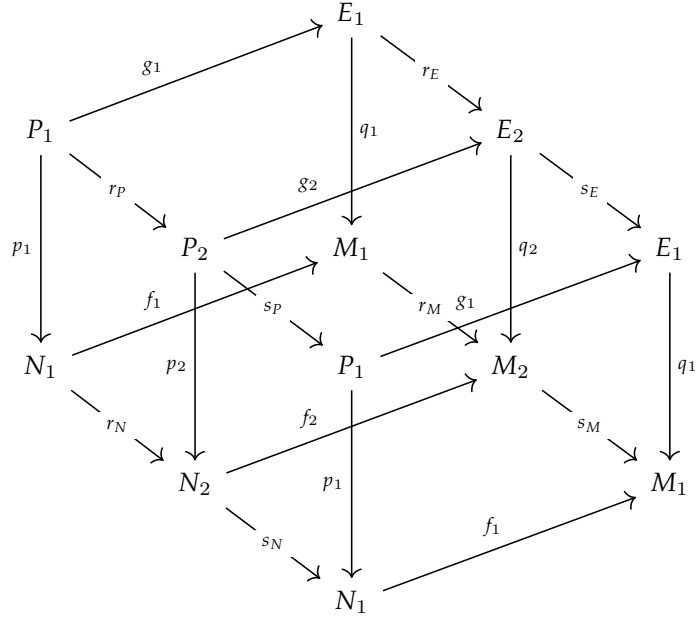
\begin{equation*}
\begin{tikzcd}
&& {E_1} \\
{P_1} &&& {E_2} \\
& {P_2} & {M_1} && {E_1} \\
{N_1} && {P_1} & {M_2} \\
& {N_2} &&& {M_1} \\
&& {N_1}
\arrow["{r_E}"{description}, from=1-3, to=2-4]
\arrow["{q_1}", from=1-3, to=3-3]
\arrow["{g_1}"{pos=0.4}, from=2-1, to=1-3]
\arrow["{r_P}"{description}, from=2-1, to=3-2]
\arrow["{p_1}"', from=2-1, to=4-1]
\arrow["{s_E}"{description}, from=2-4, to=3-5]
\arrow["{q_2}", from=2-4, to=4-4]
\arrow["{g_2}"{pos=0.4}, from=3-2, to=2-4]
\arrow["{s_P}"{description}, from=3-2, to=4-3]
\arrow["{p_2}"', from=3-2, to=5-2]
\arrow["{r_M}"{description}, from=3-3, to=4-4]
\arrow["{q_1}", from=3-5, to=5-5]
\arrow["{f_1}"{pos=0.4}, from=4-1, to=3-3]
\arrow["{r_N}"{description}, from=4-1, to=5-2]
\arrow["{g_1}"{pos=0.4}, from=4-3, to=3-5]
\arrow["{p_1}"', from=4-3, to=6-3]
\arrow["{s_M}"{description}, from=4-4, to=5-5]
\arrow["{f_2}"{pos=0.4}, from=5-2, to=4-4]
\arrow["{s_N}"{description}, from=5-2, to=6-3]
\arrow["{f_1}", from=6-3, to=5-5]
\end{tikzcd}
\end{equation*}
Suppose that $(s_M,r_M)\colon M_2\leftrightarrows M_1$, $(s_E,r_E)\colon E_2\leftrightarrows E_1$, $(s_N,r_N)\colon N_2\leftrightarrows N_1$, and $(s_P,r_P)\colon P_2\leftrightarrows P_1$ are section-retraction pairs. If the diagram:
\begin{equation}
\label{equation:back-diagram-retraction}
\begin{tikzcd}
{P_1} & {E_1} \\
{N_1} & {M_1}
\arrow["{g_1}"{pos=0.4}, from=1-1, to=1-2]
\arrow["{p_1}"', from=1-1, to=2-1]
\arrow["{q_1}", from=1-2, to=2-2]
\arrow["{f_1}"', from=2-1, to=2-2]
\end{tikzcd}
\end{equation}
is a (weak) $\T$-pullback, so is:
\begin{equation*}
\begin{tikzcd}
{P_2} & {E_2} \\
{N_2} & {M_2}
\arrow["{g_2}"{pos=0.4}, from=1-1, to=1-2]
\arrow["{p_2}"', from=1-1, to=2-1]
\arrow["{q_2}", from=1-2, to=2-2]
\arrow["{f_2}"', from=2-1, to=2-2]
\end{tikzcd}
\end{equation*}
\end{lemma}
\begin{proof}
This is a fairly straightforward exercise in basic category theory which we leave to the reader, or see \cite[Lemma~1.2.8]{macadam:vector-bundles}.
\end{proof}

\begin{proposition}
\label{proposition:T-pullback-implies-pullback}
Let us consider a tangent category $(\X,\TT)$ and suppose that the following is a (tangent) pullback diagram:
\begin{equation*}
\begin{tikzcd}
{\T P} & {\T E} \\
{\T N} & {\T M}
\arrow["{\T g}", from=1-1, to=1-2]
\arrow["{\T q'}"', from=1-1, to=2-1]
\arrow["{\T q}", from=1-2, to=2-2]
\arrow["{\T f}"', from=2-1, to=2-2]
\end{tikzcd}
\end{equation*}
Then the following is also a (tangent) pullback diagram:
\begin{equation*}
\begin{tikzcd}
P & E \\
N & M
\arrow["g", from=1-1, to=1-2]
\arrow["{q'}"', from=1-1, to=2-1]
\arrow["q", from=1-2, to=2-2]
\arrow["f"', from=2-1, to=2-2]
\end{tikzcd}
\end{equation*}
\begin{proof}
Notice that, for any object $M$, $(z,p)\colon M\leftrightarrows\T M$ constitutes a section-retraction pair. Moreover, thanks to the naturality of $z$ and $p$, the following diagram commutes:
\begin{equation*}
\begin{tikzcd}
&& {\T E} \\
{\T P} &&& E \\
& P & {\T M} && {\T E} \\
{\T N} && {\T P} & M \\
& N &&& {\T M} \\
&& {\T N}
\arrow["p"{description}, from=1-3, to=2-4]
\arrow["{\T q}", from=1-3, to=3-3]
\arrow["{\T g}"{pos=0.4}, from=2-1, to=1-3]
\arrow["p"{description}, from=2-1, to=3-2]
\arrow["{\T q'}"', from=2-1, to=4-1]
\arrow["z"{description}, from=2-4, to=3-5]
\arrow["q", from=2-4, to=4-4]
\arrow["g"{pos=0.4}, from=3-2, to=2-4]
\arrow["z"{description}, from=3-2, to=4-3]
\arrow["{q'}"', from=3-2, to=5-2]
\arrow["p"{description}, from=3-3, to=4-4]
\arrow["{\T q}", from=3-5, to=5-5]
\arrow["{\T f}", from=4-1, to=3-3]
\arrow["p"{description}, from=4-1, to=5-2]
\arrow["{\T g}"{pos=0.4}, from=4-3, to=3-5]
\arrow["{\T q'}"', from=4-3, to=6-3]
\arrow["z"{description}, from=4-4, to=5-5]
\arrow["f"', from=5-2, to=4-4]
\arrow["z"{description}, from=5-2, to=6-3]
\arrow["{\T f}"', from=6-3, to=5-5]
\end{tikzcd}
\end{equation*}
Thus, the central diagram is the retract of a (tangent) pullback and, by Lemma~\ref{lemma:pullbacks-retraction}, is also a (tangent) pullback diagram.
\end{proof}
\end{proposition}

\begin{corollary}
\label{corollary:simplication-tangent-display-maps}
In Lemma~\ref{proposition:equivalent-definition-tangent-display-maps}, the indices $n$ and $m$ can be taken both strictly positive, i.e., $n,m\geq1$.
\end{corollary}

\begin{definition}
\label{definition:split-idempotents-closed-under-pullbacks}
For a category $\X$, we say that \textbf{split idempotents are closed under pullbacks} if for any pullback diagram:
\begin{equation*}
\begin{tikzcd}
{P_1} & {E_1} \\
{N_1} & {M_1}
\arrow["{g_1}"{pos=0.4}, from=1-1, to=1-2]
\arrow["{p_1}"', from=1-1, to=2-1]
\arrow["\lrcorner"{anchor=center, pos=0.125}, draw=none, from=1-1, to=2-2]
\arrow["{q_1}", from=1-2, to=2-2]
\arrow["{f_1}"', from=2-1, to=2-2]
\end{tikzcd}
\end{equation*}
and for any section-retraction pair $(s_E,r_E)\colon E_2\leftrightarrows E_1$, the induced idempotent $e\colon P_1\to P_1$:
\begin{equation*}
\begin{tikzcd}
&& {E_1} \\
{P_1} &&& {E_2} \\
&& M && {E_1} \\
N && {P_1} & M \\
& N &&& M \\
&& N
\arrow["{r_E}"{description}, from=1-3, to=2-4]
\arrow["{q_1}", from=1-3, to=3-3]
\arrow["{g_1}"{pos=0.4}, from=2-1, to=1-3]
\arrow["\lrcorner"{anchor=center, pos=0.125}, draw=none, from=2-1, to=3-3]
\arrow["{p_1}"', from=2-1, to=4-1]
\arrow["e"{description}, dashed, from=2-1, to=4-3]
\arrow["{s_E}"{description}, from=2-4, to=3-5]
\arrow["{q_2}", from=2-4, to=4-4]
\arrow[Rightarrow, no head, from=3-3, to=4-4]
\arrow["{q_1}", from=3-5, to=5-5]
\arrow["f"{pos=0.4}, from=4-1, to=3-3]
\arrow[Rightarrow, no head, from=4-1, to=5-2]
\arrow["{g_1}"{pos=0.4}, from=4-3, to=3-5]
\arrow["\lrcorner"{anchor=center, pos=0.125}, draw=none, from=4-3, to=5-5]
\arrow["{p_1}"', from=4-3, to=6-3]
\arrow[Rightarrow, no head, from=4-4, to=5-5]
\arrow["f"{pos=0.4}, from=5-2, to=4-4]
\arrow[Rightarrow, no head, from=5-2, to=6-3]
\arrow["f", from=6-3, to=5-5]
\end{tikzcd}
\end{equation*}
splits, meaning there is an object $P_2$ and a section-retraction pair $(s_P,r_P)\colon P_2\leftrightarrows P_1$, such that $r_Ps_P=e$.
\end{definition}

\begin{lemma}
\label{lemma:split-idempotents}
If $\X$ is Cauchy complete, meaning all idempotents of $\X$ split, then split idempotents are closed under pullbacks.
\end{lemma}

\begin{proposition}
\label{proposition:retractive-display-systems}
Suppose that the split idempotents of a category $\X$ are closed under pullbacks. Then the display maps of $\X$ form a retractive display system $\Dsply(\X)$.
\begin{proof}
Consider a display map $q_1\colon E_1\to M$. We want to show that if $(s_E,r_E)\colon E_2\leftrightarrows E_1$ is a section-retraction pair, then $q_2\=s_Eq_1\colon E_2\to M$ is also a display map. Consider a morphism $f\colon N\to E$; let us prove that the pullback of $q_2$ along $f$ is well-defined. First of all, consider the induced morphism $e\colon P\to P$ making the following diagram commutative:
\begin{equation*}
\begin{tikzcd}
&& {E_1} \\
{P_1} &&& {E_2} \\
&& M && {E_1} \\
N && {P_1} & M \\
& N &&& M \\
&& N
\arrow["{r_E}"{description}, from=1-3, to=2-4]
\arrow["{q_1}", from=1-3, to=3-3]
\arrow["{g_1}"{pos=0.4}, from=2-1, to=1-3]
\arrow["\lrcorner"{anchor=center, pos=0.125}, draw=none, from=2-1, to=3-3]
\arrow["{p_1}"', from=2-1, to=4-1]
\arrow["e"{description}, dashed, from=2-1, to=4-3]
\arrow["{s_E}"{description}, from=2-4, to=3-5]
\arrow["{q_2}", from=2-4, to=4-4]
\arrow[Rightarrow, no head, from=3-3, to=4-4]
\arrow["{q_1}", from=3-5, to=5-5]
\arrow["f"{pos=0.4}, from=4-1, to=3-3]
\arrow[Rightarrow, no head, from=4-1, to=5-2]
\arrow["{g_1}"{pos=0.4}, from=4-3, to=3-5]
\arrow["\lrcorner"{anchor=center, pos=0.125}, draw=none, from=4-3, to=5-5]
\arrow["{p_1}"', from=4-3, to=6-3]
\arrow[Rightarrow, no head, from=4-4, to=5-5]
\arrow["f"{pos=0.4}, from=5-2, to=4-4]
\arrow[Rightarrow, no head, from=5-2, to=6-3]
\arrow["f", from=6-3, to=5-5]
\end{tikzcd}
\end{equation*}
It is not hard to see that $e$ is an idempotent:
\begin{align*}
=&\quad eep_1   &&(ep_1=p_1)\\
=&\quad ep_1   &&(ep_1=p_1)\\
=&\quad p_1
\end{align*}
and similarly:
\begin{align*}
=&\quad eeg_1   &&(eg_1=g_1r_Es_E)\\
=&\quad eg_1r_Es_E  &&(eg_1=g_1r_Es_E)\\
=&\quad g_1r_Es_Er_Es_E &&(s_Er_E=\id_{E_2}\\
=&\quad g_1r_Es_E
\end{align*}
Since $ee$ satisfies the same equations of $e$, i.e., $ep_1=p_1$ and $eg_1=g_qr_Es_E$, by the universality of the pullback, $ee=e$. Since split idempotents of $\X$ are closed under pullbacks, $e$ must split, meaning there exists an object $P_2$ and a section-retraction pair $(s_P,r_P)\colon P_2\leftrightarrows P_1$ such that $e=r_Ps_P$:
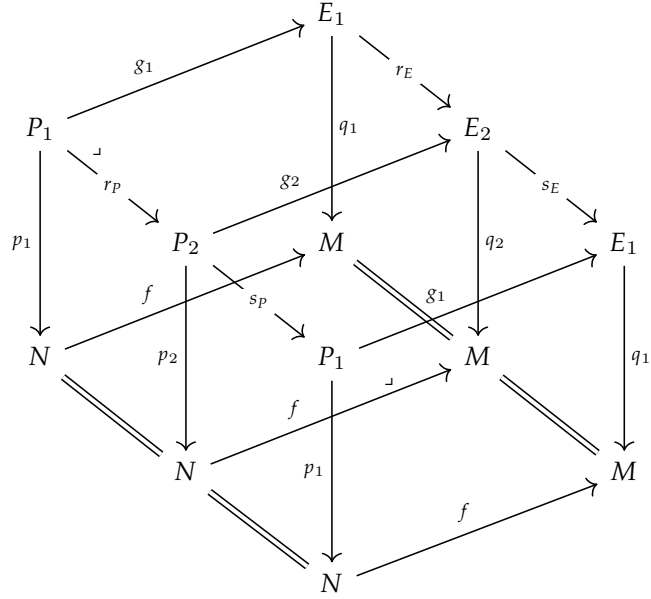
\begin{equation*}
\begin{tikzcd}
&& {E_1} \\
{P_1} &&& {E_2} \\
& {P_2} & M && {E_1} \\
N && {P_1} & M \\
& N &&& M \\
&& N
\arrow["{r_E}"{description}, from=1-3, to=2-4]
\arrow["{q_1}", from=1-3, to=3-3]
\arrow["{g_1}"{pos=0.4}, from=2-1, to=1-3]
\arrow["{r_P}"{description}, from=2-1, to=3-2]
\arrow["\lrcorner"{anchor=center, pos=0.125}, draw=none, from=2-1, to=3-3]
\arrow["{p_1}"', from=2-1, to=4-1]
\arrow["{s_E}"{description}, from=2-4, to=3-5]
\arrow["{q_2}", from=2-4, to=4-4]
\arrow["{g_2}"{pos=0.4}, from=3-2, to=2-4]
\arrow["{s_P}"{description}, from=3-2, to=4-3]
\arrow["{p_2}"', from=3-2, to=5-2]
\arrow[Rightarrow, no head, from=3-3, to=4-4]
\arrow["{q_1}", from=3-5, to=5-5]
\arrow["f"{pos=0.4}, from=4-1, to=3-3]
\arrow[Rightarrow, no head, from=4-1, to=5-2]
\arrow["{g_1}"{pos=0.4}, from=4-3, to=3-5]
\arrow["\lrcorner"{anchor=center, pos=0.125}, draw=none, from=4-3, to=5-5]
\arrow["{p_1}"', from=4-3, to=6-3]
\arrow[Rightarrow, no head, from=4-4, to=5-5]
\arrow["f"{pos=0.4}, from=5-2, to=4-4]
\arrow[Rightarrow, no head, from=5-2, to=6-3]
\arrow["f", from=6-3, to=5-5]
\end{tikzcd}
\end{equation*}
This implies that the diagram:
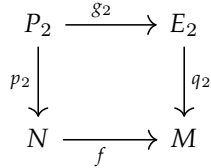
\begin{equation*}
\begin{tikzcd}
{P_2} & {E_2} \\
N & M
\arrow["{g_2}"{pos=0.4}, from=1-1, to=1-2]
\arrow["{p_2}"', from=1-1, to=2-1]
\arrow["{q_2}", from=1-2, to=2-2]
\arrow["f"'{pos=0.4}, from=2-1, to=2-2]
\end{tikzcd}
\end{equation*}
is the retraction of a pullback, where $p_2\=s_Pp_1$ and $g_2\=s_Pg_1r_E$. By Lemma~\ref{lemma:pullbacks-retraction} this diagram is also a pullback, proving that the pullback of $q_2$ along $f$ is well-defined. In particular, $q_2$ is a display map.
\end{proof}
\end{proposition}

\begin{corollary}
\label{corollary:cauchy-retractive-display}
In a Cauchy complete category, display maps form a retractive display system.
\begin{proof}
By Lemma~\ref{lemma:split-idempotents}, Cauchy completion implies that split idempotents are closed under pullbacks and Proposition~\ref{proposition:retractive-display-systems} implies that display maps are stable under retracts.
\end{proof}
\end{corollary}

\begin{theorem}
\label{theorem:-retractive-tangent-display-systems}
If the split idempotents of a tangent category $(\X,\TT)$ are closed under pullbacks, then $(\X,\TT)$ is retractive. In particular, a Cauchy complete tangent category is retractive.
\begin{proof}
The proof is essentially identical to the proof of Proposition~\ref{proposition:retractive-display-systems}, so we leave it to the reader to fill in the details.
\end{proof}
\end{theorem}

\subsection{The Cauchy completion of a tangent category}
\label{subsection:cauchy-completion}
Every category $\X$ is canonically embedded in a Cauchy complete category, called the Cauchy completion of $\X$ (the notion of a Cauchy complete category was introduced by Lawvere in~\cite{lawvere:metric-spaces-cauchy-completeness}. For more details on Cauchy completion, we refer to~\cite{borceux:cauchy-completion}). Concretely the Cauchy completion of $\X$, also known as the Karoubi envelope of $\X$, is the category $\Split(\X)$ whose objects are pairs $(M,e)$ formed by an object $M$ of $\X$ and an idempotent $e\colon M\to M$ of $M$, and whose morphisms $f\colon(M,e)\to(M',e')$ are morphisms $f\colon M\to M'$ of $\X$ for which $fe'=ef$.
\par Corollary~\ref{corollary:cauchy-retractive-display} establishes an interesting relationship between Cauchy complete categories and tangent structures, so it is natural to wonder if the Cauchy completion of a tangent category is still a tangent category. Let us start by observing a simple fact.

\begin{lemma}
Given a category $\X$, the functor $\bar{(-)}\colon\End(\X)\to\End(\Split(\X))$ which sends an endofunctor $F\colon\X\to\X$ to the endofunctor $\bar F\colon\Split(\X)\to\Split(\X)$ so defined:
\begin{align*}
&\bar F(M,e)\=(FM,Fe)\\
&\bar F(f\colon(M,e)\to(M',e'))\=Ff\colon(FM,Fe)\to(FM',Fe')
\end{align*}
is a strict monoidal functor with respect to the monoidal structure defined by the composition of endofunctors.
\end{lemma}

Using Leung's definition of a tangent category $(\X,\TT)$ (cf.~\cite{leung:weil-algebras}) as a strict monoidal functor $\Weil\to\End(\X)$ which preserves certain pullback diagrams, one can post-compose this strict monoidal functor with $\bar{(-)}$ and obtain a strict monoidal functor $\Weil\to\End(\Split(\X))$. It is not hard to show that such a strict monoidal functor preserves the required pullbacks to define a tangent structure on $\Split(\X)$.

\begin{theorem}
\label{theorem:cauchy-completion-tangent-cats}
The Cauchy completion $\Split(\X)$ of a tangent category $(\X,\TT)$ (with negatives) is still a tangent category (with negatives) denoted by $\Split(\X,\TT)$ and the fully faithful functor $\iota\colon\X\to\Split(\X)$ strictly preserves the tangent structures. Moreover, $\iota$ preserves tangent display maps and in particular, if $(\X,\TT)$ is a well-displayed tangent category (with negatives), so is $\Split(\X,\TT)$.
\end{theorem}

As a result, we also have the following:

\begin{corollary}
\label{corollary:cauchy-completion-tangent-cats}
Every (well-displayed) tangent category $(\X,\TT)$ (with negatives) is strictly embedded in a retractive (well-displayed) tangent category (with negatives).
\end{corollary}

\subsection{Tangent display maps generalize submersions}
\label{subsection:submersions}

The goal of this section is to show that the tangent display maps in the tangent category of smooth manifolds are precisely the submersions.  We begin with a general discussion of submersions and related maps in an arbitrary tangent category $(\X,\TT)$ as introduced in~\cite[Definition 1.2.5] {macadam:vector-bundles}.

\begin{definition}
\label{definition:submersions}
A \textbf{submersion} in $(\X,\TT)$ consists of a morphism $q\colon E\to M$ for which the commutative diagram:
\begin{equation}
\label{equation:naturality-projection}
\begin{tikzcd}
{\T E} & E \\
{\T M} & M
\arrow["p", from=1-1, to=1-2]
\arrow["{\T q}"', from=1-1, to=2-1]
\arrow["q", from=1-2, to=2-2]
\arrow["p"', from=2-1, to=2-2]
\end{tikzcd}
\end{equation}
is a weak tangent pullback diagram. Concretely, this means that for any pair of morphisms $\alpha\colon X\to\T M$ and $\beta\colon X\to E$ such that $\alpha p=\beta q$ there is a (not necessarily unique) morphism $\varphi\colon X\dashrightarrow\T M$ such that $\varphi\T q=\alpha$ and $\varphi p=\beta$. Moreover, all the iterates $\T^n$ of the tangent bundle functor preserve the (uni)versality property of this diagram.\\
A \textbf{display submersion} in $(\X,\TT)$ is a submersion which is also a tangent display map.
\end{definition}

One can similarly define \'etale maps in a tangent category.

\begin{definition}
\label{definition:etale-maps}
An \textbf{\'etale map} in $(\X,\TT)$ consists of a morphism $q\colon E\to M$ for which the diagram of Equation~\eqref{equation:naturality-projection} is a tangent pullback. A \textbf{display \'etale map} in $(\X,\TT)$ is an \'etale map which is also a tangent display map.
\end{definition}

In~\cite[Proposition~1.2.10]{macadam:vector-bundles}, MacAdam proved that submersions in a tangent category form a retractive tangent display system, provided that submersions are stable under pullbacks. Display submersions are the correct class of morphisms which exactly satisfy this requirement. In the next proposition, we reformulate in our language MacAdam's result and extend it to \'etale maps as well.

\begin{proposition}
\label{proposition:display-submersions-form-tangent-display-system}
The families $\Smer(\X,\TT)$ and $\Etal(\X,\TT)$ of display submersions and display \'etale maps of $(\X,\TT)$, respectively, form retractive tangent display systems of $(\X,\TT)$.
\begin{proof}
It is straightforward to see that both $\Smer(\X,\TT)$ and $\Etal(\X,\TT)$ are stable under the tangent bundle functor. Furthermore, since (weak) pullbacks are stable under retracts, it is easy to prove that submersions and \'etale maps are also stable under retracts.
\par We need to prove that they are also stable under tangent pullbacks. Let us start by considering a display submersion $q\colon E\to M$ and a morphism $f\colon N\to M$. Since $q$ is a tangent display map, the tangent pullback of $q$ along $f$ is well-defined:
\begin{equation*}
\begin{tikzcd}
P & E \\
N & M
\arrow["{\pi_2}", from=1-1, to=1-2]
\arrow["{\pi_1}"', from=1-1, to=2-1]
\arrow["\lrcorner"{anchor=center, pos=0.125}, draw=none, from=1-1, to=2-2]
\arrow["q", from=1-2, to=2-2]
\arrow["f"', from=2-1, to=2-2]
\end{tikzcd}
\end{equation*}
We want to show that $\pi_1\colon P\to N$ is still a display submersion. Since tangent display maps form a tangent display system, we already know that $\pi_1$ is a tangent display map. We need to show that $\pi_1$ is a submersion. If we have morphisms $\alpha\colon X\to\T N$ and $\beta\colon X\to P$ such that $\alpha p=\beta\pi_1$, we need to show the existence of a morphism $\varphi\colon X\to\T P$ for which $\varphi\T\pi_1=\alpha$ and $\varphi p=\beta$. Consider the following diagram:
\begin{equation*}
\begin{tikzcd}
X && {\T E} && E \\
& {\T P} && P \\
&& {\T M} && M \\
& {\T N} && N
\arrow["\beta"{pos=0.4}, bend left=20, from=1-1, to=2-4]
\arrow["\alpha"', bend right=20, from=1-1, to=4-2]
\arrow["p", from=1-3, to=1-5]
\arrow["{\T q}"{pos=0.6}, from=1-3, to=3-3]
\arrow["q", from=1-5, to=3-5]
\arrow["{\T\pi_2}", from=2-2, to=1-3]
\arrow["p"{pos=0.4}, from=2-2, to=2-4]
\arrow["{\T\pi_1}"', from=2-2, to=4-2]
\arrow["{\pi_2}", from=2-4, to=1-5]
\arrow["{\pi_1}"{pos=0.6}, from=2-4, to=4-4]
\arrow["p"{pos=0.4}, from=3-3, to=3-5]
\arrow["{\T f}"', from=4-2, to=3-3]
\arrow["p"', from=4-2, to=4-4]
\arrow["f"', from=4-4, to=3-5]
\end{tikzcd}
\end{equation*}
Since $q$ is a submersion, there exists a morphism $\psi\colon X\to\T E$ such that the following diagram commutes:
\begin{equation*}
\begin{tikzcd}
X && P \\
& {\T E} & E \\
{\T N} & {\T M} & M
\arrow["\beta", from=1-1, to=1-3]
\arrow["\psi"{description}, dashed, from=1-1, to=2-2]
\arrow["\alpha"', from=1-1, to=3-1]
\arrow["{\pi_2}", from=1-3, to=2-3]
\arrow["p", from=2-2, to=2-3]
\arrow["{\T q}"', from=2-2, to=3-2]
\arrow["q", from=2-3, to=3-3]
\arrow["{\T f}"', from=3-1, to=3-2]
\arrow["p"', from=3-2, to=3-3]
\end{tikzcd}
\end{equation*}
In particular:
\begin{align*}
&\psi\T q=\alpha\T f\\
&\psi p=\beta\pi_2
\end{align*}
Moreover, the pullback diagram of $q$ along $f$ is preserved by $\T$, therefore, there exists a unique morphism $\varphi\colon X\to\T P$ making the following diagram commutative:
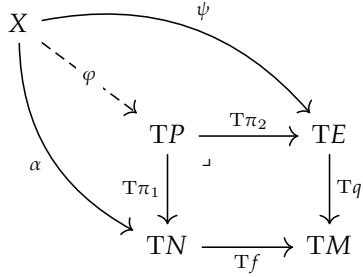
\begin{equation*}
\begin{tikzcd}
X \\
& {\T P} & {\T E} \\
& {\T N} & {\T M}
\arrow["\varphi"{description}, dashed, from=1-1, to=2-2]
\arrow["\psi", bend left, from=1-1, to=2-3]
\arrow["\alpha"', bend right, from=1-1, to=3-2]
\arrow["{\T\pi_2}", from=2-2, to=2-3]
\arrow["{\T\pi_1}"', from=2-2, to=3-2]
\arrow["\lrcorner"{anchor=center, pos=0.125}, draw=none, from=2-2, to=3-3]
\arrow["{\T q}", from=2-3, to=3-3]
\arrow["{\T f}"', from=3-2, to=3-3]
\end{tikzcd}
\end{equation*}
In particular:
\begin{align*}
&\varphi\T\pi_1=\alpha\\
&\varphi\T\pi_2=\psi
\end{align*}
We need to show that $\varphi p=\beta$. Let us employ the universality of the pullback of $q$ along $f$ and show that $\varphi p\pi_1=\beta\pi_1$ and $\varphi p\pi_2=\beta\pi_2$, since this implies the desired equation:
\begin{eqnarray*}
& &\varphi p\pi_1\\
&=&\varphi\T\pi_1p\\
&=&\alpha p\\
&=&\beta\pi_1\\
& &\\
& &\varphi p\pi_2\\
&=&\varphi\T\pi_2p\\
&=&\psi p\\
&=&\beta\pi_2
\end{eqnarray*}
This shows that $\varphi$ is the desired morphism. Let us now show that, whenever $q$ is a display \'etale map, the morphism $\varphi$ is the unique morphism such that $\varphi\T\pi_1=\alpha$ and $\varphi p=\beta$. Suppose that $\varphi'\colon X\to\T P$ is another morphism with this property. Therefore:
\begin{eqnarray*}
& &\varphi'\T\pi_2 p\\
&=&\varphi'p\pi_2\\
&=&\beta\pi_2\\
& &\\
& &\varphi'\T\pi_2\T q\\
&=&\varphi'\T\pi_1\T f\\
&=&\alpha\T f
\end{eqnarray*}
However, the same equations are satisfied by $\psi\colon X\to\T E$ and since $q$ is \'etale, $\psi$ is the unique morphism with this property. Thus, $\varphi'\T\pi_2=\psi$. However, $\varphi$ is the unique morphism such that $\varphi\T\pi_1=\alpha$ and $\varphi\T\pi_2=\psi$. Therefore, $\varphi'=\varphi$. This proves that $\pi_1$ is a display \'etale map whenever so is $q$. 
\end{proof}
\end{proposition}

Let $\Smooth$ denote the tangent category of finite-dimensional smooth manifolds.

\begin{proposition}
\label{proposition:submersions-etale-differential-geometry}
In the tangent category $\Smooth$, submersions and \'etale maps according to Definitions~\ref{definition:submersions} and~\ref{definition:etale-maps} coincide with the usual notions of submersions and \'etale maps in differential geometry, respectively. Moreover, in this tangent category, submersions and \'etale maps are automatically display submersions and display \'etale maps, respectively.
\begin{proof}
The first claim is immediate upon examining the relevant definitions, and the second claim follows from Theorem 35.13 (2) in \cite{kolavr:differential-geometry} (using the particular case of $\T_A$ being the tangent bundle functor). 
\end{proof}
\end{proposition}

Next, we would like to show that the tangent display maps in $\Smooth$ are precisely the same as the submersions.  However, it is worth making some general comments about this result first.  In particular, \cite[Lemma 71]{metzler:stacks} claims the stronger result that in the category $\Smooth$, \emph{every} display map (not just a tangent display map) is a submersion. However, some commentators have raised doubts about the proof\footnote{See, for example, mathoverflow.net/questions/71902/is-every-representable-map-a-submersion.}. 
It is not clear to us whether  \cite[Lemma 71]{metzler:stacks} is true or not; nevertheless, for our purposes, the weaker result we prove here suffices.  However, we should  note that our proof is heavily inspired by the proof of \cite[Lemma 71]{metzler:stacks}, while also using the additional assumption that pullbacks along the given map are preserved by $\T$.  

We begin with a lemma which looks at a particular class of tangent display maps in $\Smooth$.

\begin{lemma}
\label{lemma:projection-cannot-be-singular}
In the tangent category $\Smooth$, any tangent display map with codomain $\R$ is a submersion.
\begin{proof}
For contradiction, suppose that $q\colon E\to\R$ is a tangent display map which is not a submersion; that is, there is a point $a$ of $E$ for which the differential $\d_aq\colon\T_aE\to\T_{q(a)}\R$ of $q$ at $a$ is not surjective. Without loss of generality, assume that $q(a)=0$. Since $\T_0\R$ is a one-dimensional vector space and that $\d_aq$ is a linear non-surjective map, $\d_aq$ must be the zero map. Let us consider the morphism $s\colon\R\to\R$ which sends $x\in\R$ to $x^2$. Since $q$ is a tangent display map, by definition, the pullback diagram:
\begin{equation*}
\begin{tikzcd}
Q & E \\
\R & \R
\arrow["{\pi_2}", from=1-1, to=1-2]
\arrow["{\pi_1}"', from=1-1, to=2-1]
\arrow["\lrcorner"{anchor=center, pos=0.125}, draw=none, from=1-1, to=2-2]
\arrow["q", from=1-2, to=2-2]
\arrow["s"', from=2-1, to=2-2]
\end{tikzcd}
\end{equation*}
exists and is preserved by the tangent bundle functor $\T$. The underlying set of $Q$ is computed as follows:
\begin{align*}
&Q=\{(x,b)\in\R\times E,x^2=q(b)\}
\end{align*}
In particular, if $E$ has dimension $n$, $\R\times E$ has dimension $n+1$, and since $x^2=q(b)$ imposes a constraint, $Q$ has dimension $n+1-1=n$. Since $\T$ preserves the pullback of $q$ along $s$, the diagram:
\begin{equation*}
\begin{tikzcd}
{\T_{(0,a)}Q} & {\T_aE} \\
{\T_0\R} & {\T_0\R}
\arrow["{\d_{(0,a)}\pi_2}", from=1-1, to=1-2]
\arrow["{\d_{(0,a)}\pi_1}"', from=1-1, to=2-1]
\arrow["{d_aq}", from=1-2, to=2-2]
\arrow["{\d_0s}"', from=2-1, to=2-2]
\end{tikzcd}
\end{equation*}
is a pullback diagram. So, we can compute the tangent space of $Q$ at $(0,a)$ as follows:
\begin{align*}
&\T_{(0,a)}Q=\{(y,u)\in\T_0\R\times\T_aE,\d_0s(y)=\d_aq(u)\}
\end{align*}
However, $\d_0s(y)=2xy|_{x=0}=0$ and, by assumption, $\d_aq$ is also the zero map. Therefore, the equation $\d_0s(y)=\d_aq(u)$ does not impose any constraint. Thus, $\T_{(0,a)}Q\cong\T_0\R\times\T_aE\cong\R\times\T_aE$ has dimension $n+1$, contradicting that $Q$ has dimension $n$.
\end{proof}
\end{lemma}

We can now give a full characterization of submersions. 

\begin{theorem}
\label{theorem:classification-submersions}
In the tangent category $\Smooth$, the following are equivalent:
\begin{enumerate}
\item\label{class-submersion:item1} $q\colon E\to M$ is a submersion in the usual sense of differential geometry;
\item\label{class-submersion:item2} $q$ is a submersion in the sense of Definition~\ref{definition:submersions};
\item\label{class-submersion:item3} $q$ is a display submersion;
\item\label{class-submersion:item4} $q$ is a tangent display map.
\end{enumerate}
In particular, submersions in the tangent category $\Smooth$ form the maximal tangent display system which coincides with $\Smer(\Smooth)$.
\begin{proof}
In Proposition~\ref{proposition:submersions-etale-differential-geometry} we have already shown the equivalence between~\ref{class-submersion:item1},~\ref{class-submersion:item2}, and~\ref{class-submersion:item3}. It is also immediate to see that~\ref{class-submersion:item3} implies~\ref{class-submersion:item4}, since every display submersion is a tangent display map by definition. We need to show that~\ref{class-submersion:item4} implies~\ref{class-submersion:item1}.
\par Let us consider a tangent display map $q\colon E\to M$ and let $a$ be a point of $E$ and $v$ a vector in $\T_{q(a)}M$. We want to find a vector $u\in\T_aE$ for which $\d_aq(u)=v$. The tangent space $\T_xM$ of a smooth manifold $M$ at a given point $x$ of $M$ consists of the first derivatives $\gamma'(0)$ of all the regular paths $\gamma\colon\R\to M$ for which $\gamma(0)=x$. So, in particular, we can choose a regular path $\gamma\colon\R\to M$ such that $\gamma(0)=q(a)$ and $\gamma'(0)=v$. Since $q$ is a tangent display map, the tangent pullback of $q$ along $\gamma$ is well-defined.
\begin{equation*}
\begin{tikzcd}
P & E \\
\R & M
\arrow[from=1-1, to=1-2]
\arrow["{\pi_1}"', from=1-1, to=2-1]
\arrow["\lrcorner"{anchor=center, pos=0.125}, draw=none, from=1-1, to=2-2]
\arrow["q", from=1-2, to=2-2]
\arrow["\gamma"', from=2-1, to=2-2]
\end{tikzcd}
\end{equation*}
We can compute $P$ as follows:
\begin{align*}
&P=\{(t,b)\in\R\times E,\gamma(t)=q(b)\}
\end{align*}
Moreover, since the pullback is preserved by $\T$, we  have that the diagram:
\begin{equation}
\label{equation:classification-submersions}
\begin{tikzcd}
{\T_{(0,a)}P} & {\T_aE} \\
{\T_0\R} & {\T_{q(a)}M}
\arrow[from=1-1, to=1-2]
\arrow["{\d_{(0,a)}\pi_1}"', from=1-1, to=2-1]
\arrow["{\d_aq}", from=1-2, to=2-2]
\arrow["{\d_0\gamma}"', from=2-1, to=2-2]
\end{tikzcd}
\end{equation}
is also a pullback. Therefore, we can compute the tangent space of $P$ at $(0,a)$ as follows:
\begin{align*}
&\T_{(0,a)}P=\{(s,u)\in\T_0\R\times\T_aE,\d_0\gamma(s)=\d_aq(u)\}
\end{align*}
Theorem~\ref{theorem:T-display-maps-form-maximal-T-display-system} establishes that tangent display maps form a tangent display system. So, in particular, the tangent pullback of a tangent display map is again a tangent display map. This implies that $\pi_1\colon P\to\R$ is a tangent display map. Thanks to Lemma~\ref{lemma:projection-cannot-be-singular}, $\pi_1$ is a submersion. In particular, $\d_{(0,a)}\pi_1\colon\T_{(0,a)}P\to\T_0\R$ is surjective. This implies the existence of a pair $(s,u)\in\T_{(0,a)}P$ such that $\d_{(0,a)}\pi_1(s,u)=1$. However, since the diagram of Equation~\eqref{equation:classification-submersions} is a pullback, $\d_{(0,a)}\pi_1$ plays the role of the first projection, thus $\d_{(0,a)}\pi_1(s,u)=s$. Therefore, there is at least a $u\in\T_aE$ such that $(1,u)\in\T_{(0,a)}P$. Therefore:
\begin{align*}
&\d_aq(u)=\d_0\gamma(1)=1\gamma'(0)=v
\end{align*}
This concludes the proof.
\end{proof}
\end{theorem}

\section{Applications}
\label{section:applications}
In this part of the paper, we show that many previous results and ideas in tangent categories can be simplified by asking that certain maps be tangent display maps.

\subsection{Display differential bundles}
\label{subsection:display-differential-bundles}
In differential geometry, vector bundles are an important structure.  Loosely speaking, a vector bundle consists of a projection $q\colon E\to M$ whose fibres carry the structure of a finite-dimensional vector space. The tangent categorical analogs of vector bundles are differential bundles, first introduced by Cockett and the first author in~\cite{cockett:tangent-cats}.

\begin{definition}
\label{definition:differential-bundles}
A \textbf{differential bundle} in a tangent category $(\X,\TT)$ consists of an additive bundle $(q\colon E\to M,z_q\colon M\to E,s_q\colon E_2\to E)$ together with a morphism $l_q\colon E\to\T E$, called the \textbf{vertical lift}, fulfilling the following conditions:
\begin{enumerate}
\item $(l_q,z)\colon(q,z_q,s_q)\to(\T q,\T z_q,\T s_q)$ is an additive bundle morphism;

\item $(l_q,z_q)\colon(q,z_q,s_q)\to(p,z,s)$ is an additive bundle morphism;

\item The vertical lift is universal, that is, the following diagram:
\begin{equation*}
\begin{tikzcd}
{E_2} & {\T E} \\
M & {\T M}
\arrow["\iota_q", from=1-1, to=1-2]
\arrow["{\pi_1q}"', from=1-1, to=2-1]
\arrow["{z_q}"', from=2-1, to=2-2]
\arrow["{\T q}", from=1-2, to=2-2]
\end{tikzcd}
\end{equation*}
is a tangent pullback diagram, where:
\begin{align*}
&\iota_q\=(l_q\times_Mz)\T s_q
\end{align*}

\item The vertical lifts $l$ and $l_q$ are compatible:
\begin{equation*}
\begin{tikzcd}
E & {\T E} \\
{\T E} & {\T^2 E}
\arrow["{l_q}", from=1-1, to=1-2]
\arrow["{l_q}"', from=1-1, to=2-1]
\arrow["l"', from=2-1, to=2-2]
\arrow["{\T l_q}", from=1-2, to=2-2]
\end{tikzcd}
\end{equation*}
\end{enumerate}
\end{definition}

The interpretation of differential bundles as vector bundles in a tangent category acquires solidity in light of MacAdam's result, presented in~\cite[Theorem 4.2.7]{macadam:vector-bundles}, which proves that differential bundles in the tangent category of (connected) smooth manifolds are precisely vector bundles.
\par In \cite[Corollary~3.5]{cockett:differential-bundles}, Cockett and the first author also showed that given a point $x\colon\1\to M$, i.e., a morphism from the terminal object $\1$ of $\X$ to the object $M$, the local fibre $E_x$ of a differential bundle $q\colon E\to M$ over $x$, obtained by pulling back $q\colon E\to M$ along $x\colon\1\to M$ as follows:
\begin{equation*}
\begin{tikzcd}
{E_x} & E \\
\1 & M
\arrow["x"', from=2-1, to=2-2]
\arrow["q", from=1-2, to=2-2]
\arrow[from=1-1, to=2-1]
\arrow[from=1-1, to=1-2]
\arrow["\lrcorner"{anchor=center, pos=0.125}, draw=none, from=1-1, to=2-2]
\end{tikzcd}
\end{equation*}
is a differential object. 
\par The connection betwen differential bundles and differential objects is even deeper: under certain conditions, differential bundles are precisely the differential objects in the slice tangent category (see Section~\ref{subsection:slice-tangent-category}). However, to define the correct tangent structure on the slice category certain pullbacks are required to exist. Tangent display systems were initially introduced specifically to solve this issue. With the new notion of tangent display maps, we can instead require that differential bundles be tangent display maps.

\begin{definition}
\label{definition:display-differential-bundles}
A \textbf{display differential bundle} in a tangent category $(\X,\TT)$ consists of a differential bundle whose projection is a tangent display map.
\end{definition}

\begin{remark}
In \cite[Definition 4.21]{cockett:differential-bundles}, a display differential bundle in a display tangent category is a differential bundle whose projection belongs to the given display system.  Here we are defining a display differential bundle in an arbitrary tangent category as one in which the projection is a tangent display map (that is, part of the maximal display system).  
\end{remark}

\cite[Lemma~2.7]{cockett:differential-bundles} shows that the tangent pullback of a differential bundle along a morphism is still a differential bundle. This makes display differential bundles closed under tangent pullbacks. Furthermore, \cite[Lemma~2.5]{cockett:differential-bundles} shows that the tangent bundle functor sends differential bundles to differential bundles. These two results make the family of differential bundles of a tangent category into a tangent display system.

\begin{proposition}
\label{proposition:display-differential-bundles-form-tangent-display-system}
The family of display differential bundles in a tangent category $(\X,\TT)$ forms a tangent display system $\DBnd(\X,\TT)$.
\end{proposition}

In a well-displayed tangent category every tangent bundle $p\colon\T M\to M$ is a tangent display map and therefore a display differential bundle. However, it is not obvious that an arbitrary differential bundle would be display as well. Let us introduce this concept.

\begin{definition}
\label{definition:fully-display-tangent-category}
A \textbf{fully displayed tangent category} is a tangent category in which every differential bundle is a display differential bundle. 
\end{definition}

In particular, every fully displayed tangent category is a well-displayed tangent category. In~\cite[Corollary~3.1.4]{macadam:vector-bundles}, MacAdam proved an important characterization of differential bundles: when the tangent category has negatives, every differential bundle is the retract of the pullback of a tangent bundle. We can employ this characterization to show that when the tangent bundles are tangent display maps, the tangent category has negatives, and the tangent display maps are stable under retracts, then every differential bundle automatically becomes a tangent display map. Let us first recall MacAdam's result.

\begin{lemma}
\label{lemma:differential-bundles-macadam}\cite[Corollary~3.1.4]{macadam:vector-bundles}
In a tangent category with negatives, a display differential bundle $q\colon E\to M$ is the retract of a pullback of the tangent bundle $p\colon\T E\to E$. Concretely, there exists a section-retraction pair $(\<l_q,q\>,k)\colon E\leftrightarrows VE$, where $VE\to M$ denotes the vertical bundle of $q$, which is the pullback of $\T q$ along the zero morphism $z\colon M\to\T M$, and $l_q\colon E\to\T E$ denotes the vertical lift of $q$.
\end{lemma}

\begin{theorem}
\label{theorem:fully-display-tangent-cats}
A retractive well-displayed tangent category with negatives is a fully displayed tangent category.
\begin{proof}
By definition, tangent display maps of a retractive tangent category $(\X,\TT)$ form a retractive tangent display system. Therefore, $\Dsply(\X,\TT)$ is stable under retracts and under pullbacks. By Lemma~\ref{lemma:differential-bundles-macadam}, differential bundles of a tangent category with negatives are retracts of pullbacks of tangent bundles. Finally, by definition, the tangent bundles of a well-displayed tangent category are tangent display maps. Therefore, every differential bundle is display, meaning the tangent category is fully displayed.
\end{proof}
\end{theorem}

Thanks to Corollary~\ref{corollary:cauchy-retractive-display}, every Cauchy complete tangent category is retractive, so we also have the following result.

\begin{corollary}
\label{corollary:fully-display-tangent-cats}
A Cauchy complete well-displayed tangent category with negatives is a fully displayed tangent category with negatives.
\end{corollary}

Finally, putting together that the Cauchy completion of a well-displayed tangent category with negatives is still a well-displayed tangent category with negatives (Theorem~\ref{theorem:cauchy-completion-tangent-cats} and Corollary~\ref{corollary:cauchy-completion-tangent-cats}) and using Corollary~\ref{corollary:fully-display-tangent-cats}, we also obtain the following result.

\begin{corollary}
\label{corollary:fully-display-embedding}
Every well-displayed tangent category $(\X,\TT)$ with negatives is embedded in a fully displayed tangent category $\Split(\X,\TT)$ with negatives.
\end{corollary}

\subsection{Tangent fibrations}
\label{subsection:tangent-fibrations}
One of the motivations for introducing tangent display systems in \cite{cockett:differential-bundles} was to organize differential bundles of a given tangent category into a fibration, satisfying some compatibility conditions with the tangent bundle functor. \cite[Section~5]{cockett:differential-bundles} showed that each tangent category equipped with a choice of a tangent display system $\Dsply$ defines a fibration whose fibre over an object $M$ is the category of the maps of $\Dsply$ whose codomain is $M$. Furthermore, such a fibration is compatible with the tangent bundle functors.
\par Cockett and the first author distilled the property of this class of fibrations into the notion of a tangent fibration. This section is dedicated to recalling this definition and discussing the relationship between tangent display maps and the corresponding tangent fibration.
\par Let us start by briefly recalling that a fibration $\Pi\colon\X'\to\X$ between two categories consists of a functor for which every morphism $f\colon M\to M'$ of the base category $\X$ admits a cartesian lift $\varphi_f^{E'}\colon f^\*E'\to E'$ for every object $E'$ on the fibre over $M'$, that is $\Pi(E')=M'$. Concretely, a cartesian morphism $\varphi\colon E'\to E''$ is a morphism of the total category $\X'$ such that for any morphism $\psi\colon E\to E''$ of $\X'$ and any morphism $g\colon\Pi(E)\to\Pi(E')$ making the following diagram commutes:
\begin{equation*}
\begin{tikzcd}
& {\Pi(E')} \\
{\Pi(E)} && {\Pi(E'')}
\arrow["{\Pi(\varphi)}", from=1-2, to=2-3]
\arrow["g", from=2-1, to=1-2]
\arrow["{\Pi(\psi)}"', from=2-1, to=2-3]
\end{tikzcd}
\end{equation*}
there exists a unique morphism $\xi\colon E\to E'$ of $\X'$ such that $\Pi(\xi)=g$ and the following diagram commutes:
\begin{equation*}
\begin{tikzcd}
& {E'} \\
E && {E''}
\arrow["\varphi", from=1-2, to=2-3]
\arrow["\xi", dashed, from=2-1, to=1-2]
\arrow["\psi"', from=2-1, to=2-3]
\end{tikzcd}
\end{equation*}
A fibration $\Pi\colon\X'\to\X$ equipped with a choice of a cartesian lift $\varphi_f^{E'}\colon f^\*E'\to E'$ for each morphism $f\colon M\to M'$ of $\X$ and each object $E'$ of the fibre over $M'$ is known as a \textbf{cloven fibration} and the choice $(f,E')\mapsto\varphi_f^{E'}$ is called a \textbf{cleavage} of the fibration.
\par In the following, every fibration is assumed to be cloven and we denote by $\varphi_f^{E'}$ the cartesian lift of $f\colon M\to M'$ along $E'\in\Pi^{-1}(M')$ defined by the cleavage. When $E'$ is clear from the context we omit the superscript $E'$. Moreover, we also denote a fibration $\Pi\colon\X'\to\X$ as a triple $(\X,\X',\Pi)$.\newline
\par Suppose $\Pi_\o\colon\X_\o'\to\X_\o$ and $\Pi_\b\colon\X_\b'\to\X_\b$ are two fibrations and $F'\colon\X'_\o\to\X'_\b$ and $F\colon\X_\o\to\X_\b$ are two functors such that $\Pi_\b\o F'=F\o\Pi_\o$. Given a morphism $f\colon M\to M'$ of $\X_\o$, one can lift $Ff\colon FM\to FM'$ to $\X_\b'$ using the cleavage of $\Pi_\b$ and obtain a cartesian lift $\varphi_{Ff}\colon(Ff)^\*(F'E')\to F'E'$ of $Ff$ over $F'E'\in\X_\b'$. \par However, the cleavage of $\Pi_\o$ also defines a cartesian lift $\varphi_f\colon f^\*E'\to E'$ of $f$ over $E'$. Therefore, since $\Pi_\b F'\varphi_f=F\Pi_\o\varphi_f=Ff$, $F'\varphi_f$ is a lift of $Ff$. By employing the universal property of the cartesian lift $\varphi_{Ff}$, there exists a unique morphism $\xi_f\colon F'(f^\*E')\dashrightarrow(Ff)^\*(F'E')$ making the following diagram commute:
\begin{equation*}
\begin{tikzcd}
& {(Ff)^\*(F'E')} \\
{F'(f^\*E')} && {F'E'}
\arrow["{\varphi_{Ff}}", from=1-2, to=2-3]
\arrow["{\xi_f}", dashed, from=2-1, to=1-2]
\arrow["{F'\varphi_f}"', from=2-1, to=2-3]
\end{tikzcd}
\end{equation*}
A morphism of fibrations $(F,F')\colon(\X_\o,\X_\o',\Pi_\o)\to(\X_\b,\X_\b',\Pi_\b)$ consists of a pair of functors $F\colon\X_\o\to\X_\b$ and $F'\colon\X_\o'\to\X_\b'$ which commute strictly with the two fibrations, that is $\Pi_\b\o F'=F\o\Pi_\o$ and which preserve the cartesian lifts, meaning the unique morphism $\xi_f\colon F'(f^\*E')\dashrightarrow(Ff)^\*(F'E')$ induced by the universality of $\varphi_{Ff}$ is an isomorphism for each morphism $f\colon M\to M'$ of $\X_\o$ and every object $E'\in\Pi_\o^{-1}(M')$.

\begin{definition}
\label{definition:tangent-fibration}
Given two tangent categories $(\X',\TT')$ and $(\X,\TT)$, a \textbf{tangent fibration} $\Pi\colon(\X',\TT')\to(\X,\TT)$ consists of a fibration $\Pi\colon\X'\to\X$ whose underlying functor is a strict tangent morphism and such that the tangent bundle functors preserve the cartesian lifts, meaning $(\T,\T')\colon(\X,\X',\Pi)\to(\X,\X',\Pi)$ is a morphism of fibrations.
\end{definition}

Cockett and the first author showed that a tangent display system $\Dsply$ of a tangent category $(\X,\TT)$ gives rise to a tangent category $(\Dsply,\TT')$ whose tangent bundle functor sends each map $q\colon E\to M$ of $\Dsply$ to $\T q\colon\T E\to\T M$. Furthermore, the functor which sends each map $q\colon E\to M$ to its codomain $M$ defines a tangent fibration $(\Dsply,\TT')\to(\X,\TT)$. Here, we recall this result and apply it to the setting of tangent display maps.

\begin{proposition}
\label{proposition:codomain-tangent-fibration}
The functor $\Pi\colon\Dsply\to\X$ which sends each map $q\colon E\to M$ of a tangent display system $\Dsply$ of a tangent category $(\X,\TT)$ to its codomain $M$ defines a tangent fibration $\Pi\colon(\Dsply,\TT')\to(\X,\TT)$. In particular, the maximal tangent display system $\Dsply(\X,\TT)$ of the tangent display maps of $(\X,\TT)$ defines a tangent fibration $\Pi\colon\Dsply(\X,\TT)\to(\X,\TT)$.
\end{proposition}

By applying Proposition~\ref{proposition:codomain-tangent-fibration} to the tangent display system $\DBnd(\X,\TT)$ of display differential bundles (see Proposition~\ref{proposition:display-differential-bundles-form-tangent-display-system}) one also defines a tangent fibration $\DBnd(\X,\TT)\to(\X,\TT)$. Similarly, one can also apply the same construction to the tangent display systems $\Smer(\X,\TT)$ and $\Etal(\X,\TT)$ of display submersions and display \'etale maps of $(\X,\TT)$ to get two additional tangent fibrations $\Smer(\X,\TT)\to(\X,\TT)$ and $\Etal(\X,\TT)\to(\X,\TT)$.

\subsection{The universal property of slicing}
\label{subsection:slice-tangent-category}
The slice category $\X/M$ of a category $\X$ over an object $M$ of $\X$ is the category whose objects $q^E_M$ are morphisms of $\X$ of the form $q\colon E\to M$ and whose morphisms $f\colon q^E_M\to {q'}^{E'}_M$ are morphisms $f\colon E\to E'$ of $\X$ for which $fq'=q$. When the base category $\X$ is equipped with a tangent structure, one would like to lift the tangent structure $\TT$ of $\X$ to the slice category $\X/M$.
\par One approach is to define the tangent bundle functor $\T/M\colon\X/M\to\X/M$ as the functor which sends a bundle $q^E_M$ to:
\begin{align*}
\T q/M\colon\T E\xrightarrow{\T q}\T M\xrightarrow{p}M
\end{align*}
The structural natural transformations of this tangent structure are precisely the same as the ones of the tangent structure $\TT$. We refer to this tangent structure as the ``\textbf{trivial slice tangent structure}''.
\par A second, and more interesting construction, is the ``non-trivial'' slice tangent category. This plays an important role in tangent category theory. In particular, (display) differential bundles are equivalent to differential objects in the non-trivial slice tangent category. For its importance, in the following, we will refer to this as \emph{the} slice tangent category.
\par The idea is to define the tangent bundle of $q\colon E\to M$ as the pullback of $\T q$ along the zero morphism, that is the bundle $q\^M\colon\T\^ME\to M$:
\begin{equation}
\label{equation:sliceability}
\begin{tikzcd}
{\T\^ME} & {\T E} \\
M & {\T M}
\arrow["{q\^M}"', from=1-1, to=2-1]
\arrow["{\T q}", from=1-2, to=2-2]
\arrow["z"', from=2-1, to=2-2]
\arrow["\iota", from=1-1, to=1-2]
\arrow["\lrcorner"{anchor=center, pos=0.125}, draw=none, from=1-1, to=2-2]
\end{tikzcd}
\end{equation}
Notice that $\T\^-$ is a functor on the slice category but not on the base category.
\par Unfortunately, in an arbitrary tangent category there is no need for the pullback diagram of Equation~\eqref{equation:sliceability} to exist for an arbitrary morphism $q\colon E\to M$. Furthermore, in order to properly define the desired tangent structure on the slice category, this pullback should also be a tangent pullback and the pulled-back morphism should also admit the corresponding tangent pullback\footnote{The definition of these two tangent structures on the slice is due to Rosick\`{y} \cite{rosicky:tangent-cats}. However, no mention there is made of the pullback assumptions needed to define the second slice tangent structure.}.
\par One approach is to focus only on objects $M$ of the category $\X$ for which every $q\in\X/M$ admits all tangent pullbacks. In~\cite{lanfranchi:differential-bundles-operadic-affine-schemes}, the second author employed this approach. However, this has an important disadvantage: in general strong tangent morphisms do not preserve this property on the objects. Consequently, the slice tangent category so defined over an object $M$ cannot be obtained as the fibre over $M$ of a tangent fibration.
\par This obstruction led us to rethink the definition of the slice tangent category. This section is dedicated to presenting this new approach.

\begin{lemma}
\label{lemma:display-maps-sliceable}
If $q\colon E\to M$ is a morphism of a tangent display system $\Dsply$ in a tangent category $(\X,\TT)$, then the tangent pullback of Equation~\eqref{equation:sliceability} is well-defined. Moreover, for each $q$ in the tangent display system, $q\^M$ is a morphism of the same tangent display system.
\begin{proof}
Since, by definition, a tangent display system is stable under the tangent bundle functor, for every bundle $q$ of the tangent display system, $\T q$ is also part of the tangent display system. Moreover, since a tangent display system is also stable under pullbacks, then the pullback of $\T q$ along the zero morphism is well-defined and it defines a new morphism of the tangent display system. Thus, also $\T\^Mq=q\^M$ is part of the tangent display system. Finally, since every bundle of a tangent display system is necessarily a tangent display map, the pullback diagram is a tangent pullback.
\end{proof}
\end{lemma}

In~\cite[Theorem~5.3]{cockett:differential-bundles}, Cockett and the first author showed that the fibres of a tangent fibration come equipped with a tangent structure, strongly preserved by the substitution functors between the fibres. In Proposition~\ref{proposition:codomain-tangent-fibration}, we employed another result of Cockett and the first author to define a tangent fibration $\Dsply(\X,\TT)\to(\X,\TT)$ for each tangent category $(\X,\TT)$.
\par The \textbf{display slice tangent category} $\Dsply(\X,\TT;M)$ of $(\X,\TT)$ over an object $M$ is the tangent category given by the fibre of $\Dsply(\X,\TT)\to(\X,\TT)$ over $M$. Explicitly, it is defined as follows:
\begin{description}
\item[objects] The objects are tangent display maps $q^E_M\colon E\to M$ whose codomain is the object $M$;

\item[morphisms] The morphisms $f\colon q^E_M\to {q'}^{E'}_M$ are morphisms $f\colon E\to E'$ of $\X$ such that $f{q'}^{E'}_M=q^E_M$;

\item[tangent bundle functor] The tangent bundle functor $\T\^M\colon\Dsply(\X,\TT;M)\to\Dsply(\X,\TT;M)$ 
sends a tangent display map $q\colon E\to M$ to $q\^M$, defined by the pullback diagram~\eqref{equation:sliceability}. Moreover, $\T\^M$ sends a morphism $f\colon q^E_M\to{q'}^{E'}_M$ to the unique morphism $f\^M\colon q\^M\to{q'}\^M$, induced by the universality of the pullback diagram~\eqref{equation:sliceability}:
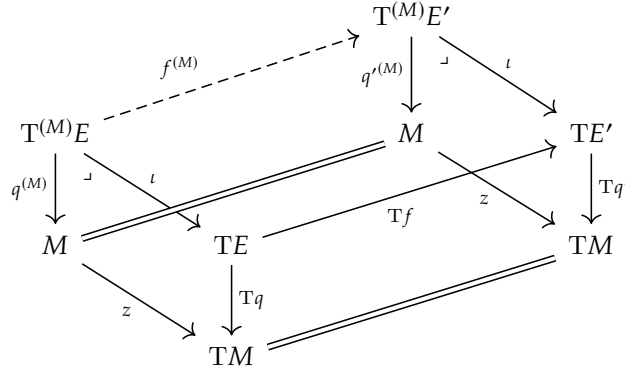
\begin{equation*}
\begin{tikzcd}
&& {\T\^ME'} \\
{\T\^ME} && M & {\T E'} \\
M & {\T E} && {\T M} \\
& {\T M}
\arrow["{{q'}\^M}"', from=1-3, to=2-3]
\arrow["\iota", from=1-3, to=2-4]
\arrow["\lrcorner"{anchor=center, pos=0.125}, draw=none, from=1-3, to=3-4]
\arrow["{f\^M}"{pos=0.4}, dashed, from=2-1, to=1-3]
\arrow["{q\^M}"', from=2-1, to=3-1]
\arrow["\iota", from=2-1, to=3-2]
\arrow["\lrcorner"{anchor=center, pos=0.125}, draw=none, from=2-1, to=4-2]
\arrow["z"', from=2-3, to=3-4]
\arrow["{\T q'}", from=2-4, to=3-4]
\arrow[Rightarrow, no head, from=3-1, to=2-3]
\arrow["z"', from=3-1, to=4-2]
\arrow["{\T f}"'{pos=0.4}, from=3-2, to=2-4]
\arrow["{\T q}", from=3-2, to=4-2]
\arrow[Rightarrow, no head, from=4-2, to=3-4]
\end{tikzcd}
\end{equation*}

\item[projection] The projection $p\^M\colon\T\^M\Rightarrow\id_{\X/M}$ is induced by the natural transformation:
\begin{align*}
&\T\^ME\xrightarrow{\iota}\T E\xrightarrow{p}E
\end{align*}
for any object $q^E_M$;

\item[zero morphism] The zero morphism $z\^M\colon\id_{\X/M}\Rightarrow\T\^M$ is induced by the natural transformation defined by the universality of the pullback diagram~\eqref{equation:sliceability}:
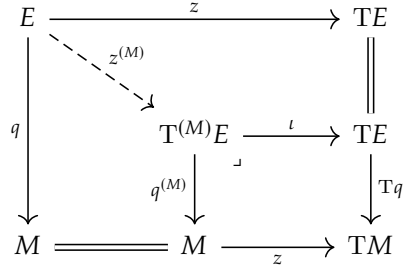
\begin{equation*}
\begin{tikzcd}
E && {\T E} \\
& {\T\^ME} & {\T E} \\
M & M & {\T M}
\arrow["{q\^M}"', from=2-2, to=3-2]
\arrow["{\T q}", from=2-3, to=3-3]
\arrow["z"', from=3-2, to=3-3]
\arrow["\iota", from=2-2, to=2-3]
\arrow["\lrcorner"{anchor=center, pos=0.125}, draw=none, from=2-2, to=3-3]
\arrow["q"', from=1-1, to=3-1]
\arrow[Rightarrow, no head, from=3-1, to=3-2]
\arrow["z", from=1-1, to=1-3]
\arrow[Rightarrow, no head, from=1-3, to=2-3]
\arrow["{z\^M}", dashed, from=1-1, to=2-2]
\end{tikzcd}
\end{equation*}

\item[$n$-fold pullback] The $n$-fold pullback $\T\^M_n$ of the projection $p\^M$ along itself is given by the pullback diagram:
\begin{equation*}
\begin{tikzcd}
{\T\^M_nE} & {\T_nE} \\
M & {\T_nM}
\arrow["{\<\iota\,\iota\>}", from=1-1, to=1-2]
\arrow["{q\^M_n}"', from=1-1, to=2-1]
\arrow["\lrcorner"{anchor=center, pos=0.125}, draw=none, from=1-1, to=2-2]
\arrow["{\T_nq}", from=1-2, to=2-2]
\arrow["{\<z\,z\>}"', from=2-1, to=2-2]
\end{tikzcd}
\end{equation*}
meaning that:
\begin{align*}
&\T\^M_nq^E_M=q\^M_n
\end{align*}
and the $k$-th projection $\pi\^M_k\colon\T\^M_n\Rightarrow\T\^M$ is given by the natural transformation induced by the universality of the following diagram:
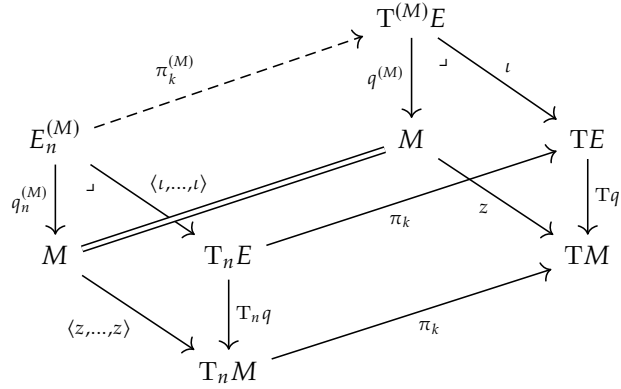
\begin{equation*}
\begin{tikzcd}
&& {\T\^ME} \\
{E_n\^M} && M & {\T E} \\
M & {\T_nE} && {\T M} \\
& {\T_nM}
\arrow["{q_n\^M}"', from=2-1, to=3-1]
\arrow["{\T_nq}", from=3-2, to=4-2]
\arrow["{\<z\,z\>}"', from=3-1, to=4-2]
\arrow["{\<\iota\,\iota\>}", from=2-1, to=3-2]
\arrow["\lrcorner"{anchor=center, pos=0.125}, draw=none, from=2-1, to=4-2]
\arrow["{\pi_k}"', from=4-2, to=3-4]
\arrow["{\pi_k}"'{pos=0.4}, from=3-2, to=2-4]
\arrow["{\T q}", from=2-4, to=3-4]
\arrow["\iota", from=1-3, to=2-4]
\arrow["{\pi\^M_k}"{pos=0.4}, dashed, from=2-1, to=1-3]
\arrow["{q\^M}"', from=1-3, to=2-3]
\arrow["z"', from=2-3, to=3-4]
\arrow[Rightarrow, no head, from=3-1, to=2-3]
\arrow["\lrcorner"{anchor=center, pos=0.125}, draw=none, from=1-3, to=3-4]
\end{tikzcd}
\end{equation*}

\item[sum morphism] The sum morphism $s\^M\colon\T\^M_2\Rightarrow\T\^M$ is induced by the natural transformation defined by the universality of the pullback diagram~\eqref{equation:sliceability}:
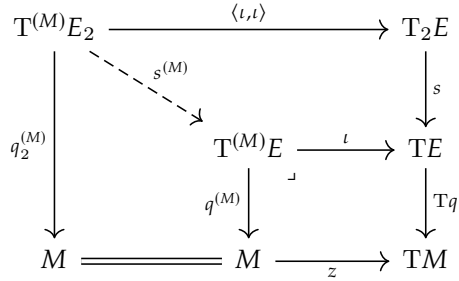
\begin{equation*}
\begin{tikzcd}
{\T\^ME_2} && {\T_2E} \\
& {\T\^ME} & {\T E} \\
M & M & {\T M}
\arrow["{\T q}", from=2-3, to=3-3]
\arrow["z"', from=3-2, to=3-3]
\arrow["s", from=1-3, to=2-3]
\arrow["{\<\iota,\iota\>}", from=1-1, to=1-3]
\arrow["{q\^M_2}"', from=1-1, to=3-1]
\arrow["{q\^M}"', from=2-2, to=3-2]
\arrow[Rightarrow, no head, from=3-1, to=3-2]
\arrow["\iota", from=2-2, to=2-3]
\arrow["{s\^M}", dashed, from=1-1, to=2-2]
\arrow["\lrcorner"{anchor=center, pos=0.125}, draw=none, from=2-2, to=3-3]
\end{tikzcd}
\end{equation*}

\item[vertical lift] The vertical lift $l\^M\colon\T\^M\Rightarrow{\T\^M}^2$ is induced by the natural transformation defined by the universality of the pullback diagram~\eqref{equation:sliceability}:
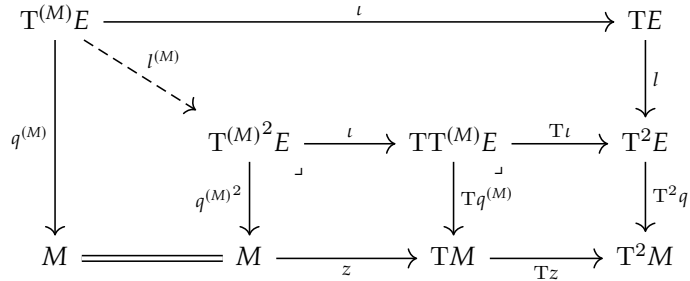
\begin{equation*}
\begin{tikzcd}
{\T\^ME} &&& {\T E} \\
& {{\T\^M}^2E} & {\T\T\^ME} & {\T^2E} \\
M & M & {\T M} & {\T^2M}
\arrow["\iota", from=1-1, to=1-4]
\arrow["{l\^M}", dashed, from=1-1, to=2-2]
\arrow["{q\^M}"', from=1-1, to=3-1]
\arrow["l", from=1-4, to=2-4]
\arrow["\iota", from=2-2, to=2-3]
\arrow["{{q\^M}^2}"', from=2-2, to=3-2]
\arrow["\lrcorner"{anchor=center, pos=0.125}, draw=none, from=2-2, to=3-3]
\arrow["{\T\iota}", from=2-3, to=2-4]
\arrow["{\T q\^M}", from=2-3, to=3-3]
\arrow["\lrcorner"{anchor=center, pos=0.125}, draw=none, from=2-3, to=3-4]
\arrow["{\T^2q}", from=2-4, to=3-4]
\arrow[Rightarrow, no head, from=3-1, to=3-2]
\arrow["z"', from=3-2, to=3-3]
\arrow["{\T z}"', from=3-3, to=3-4]
\end{tikzcd}
\end{equation*}

\item[canonical flip] The canonical flip $c\^M\colon{\T\^M}^2\Rightarrow{\T\^M}^2$ is induced by the natural transformation defined by the universality of the pullback diagram~\eqref{equation:sliceability}:
\begin{equation*}
\begin{tikzcd}
{{\T\^M}^2E} && {\T\^{\T M}\T E} & {\T^2E} \\
& {{\T\^M}^2E} & {\T\T\^ME} & {\T^2E} \\
M & M & {\T M} & {\T^2M}
\arrow["\iota", from=1-1, to=1-3]
\arrow["{c\^M}", dashed, from=1-1, to=2-2]
\arrow["{{q\^M}^2}"', from=1-1, to=3-1]
\arrow["{\iota\T}", from=1-3, to=1-4]
\arrow["c", from=1-4, to=2-4]
\arrow["\iota", from=2-2, to=2-3]
\arrow["{{q\^M}^2}"', from=2-2, to=3-2]
\arrow["\lrcorner"{anchor=center, pos=0.125}, draw=none, from=2-2, to=3-3]
\arrow["{\T\iota}", from=2-3, to=2-4]
\arrow["{\T q\^M}", from=2-3, to=3-3]
\arrow["\lrcorner"{anchor=center, pos=0.125}, draw=none, from=2-3, to=3-4]
\arrow["{\T^2q}", from=2-4, to=3-4]
\arrow[Rightarrow, no head, from=3-1, to=3-2]
\arrow["z"', from=3-2, to=3-3]
\arrow["{\T z}"', from=3-3, to=3-4]
\end{tikzcd}
\end{equation*}
\end{description}
Moreover, if $(\X,\TT)$ has negatives with negation $n\colon\T\Rightarrow\T$, so does $\Dsply(\X,\TT;M)$:
\begin{description}
\item[negation] The negation $n\^M\colon\T\^M\Rightarrow\T\^M$ is induced by the natural transformation defined by the universality of the pullback diagram~\eqref{equation:sliceability}:
\begin{equation*}
\begin{tikzcd}
{\T\^ME} && {\T_2E} \\
& {\T\^ME} & {\T E} \\
M & M & {\T M}
\arrow["{q\^M}"', from=2-2, to=3-2]
\arrow["z"', from=3-2, to=3-3]
\arrow["\lrcorner"{anchor=center, pos=0.125}, draw=none, from=2-2, to=3-3]
\arrow["\iota", from=2-2, to=2-3]
\arrow["{\T q}", from=2-3, to=3-3]
\arrow["n", from=1-3, to=2-3]
\arrow[Rightarrow, no head, from=3-1, to=3-2]
\arrow["\iota", from=1-1, to=1-3]
\arrow["{q\^M}"', from=1-1, to=3-1]
\arrow["{n\^M}", dashed, from=1-1, to=2-2]
\end{tikzcd}
\end{equation*}
\end{description}
This construction was first introduced by Rosick\`y in his seminal article~\cite{rosicky:tangent-cats}. More recently in~\cite[Section~5]{cockett:differential-bundles}, Cockett and Cruttwell showed how this construction is naturally contextualized within the theory of tangent fibrations. In particular, they proved that the fibres of the functor $\Dsply(\X,\TT)\to(\X,\TT)$ are precisely the slice tangent categories $\Dsply(\X,\TT;M)$, parametrized by the objects $M$ of $\X$. This result inspired the second author to investigate the relationship between tangent fibrations and the celebrated Grothendieck construction; for more details, see~\cite{lanfranchi:grothendieck-tangent-cats}.

\begin{remark}
\label{remark:sliceability-Rosicky-version}
In Rosick\`y's original version for the construction of the slice tangent category the pullback diagram ~\eqref{equation:sliceability} was replaced by the equalizer diagram:
\begin{equation*}
\begin{tikzcd}
{\T\^ME} & {\T E} && {\T M}
\arrow["\iota", dashed, from=1-1, to=1-2]
\arrow["{\T q}", shift left=2, from=1-2, to=1-4]
\arrow["{\T q pz}"', shift right=2, from=1-2, to=1-4]
\end{tikzcd}
\end{equation*}
So, in Rosick\`y's version, the tangent bundle functor $\T\^M$ sends $q^E_M$ to:
\begin{align*}
&\T\^ME\xrightarrow{\iota}\T E\xrightarrow{\T q}\T M\xrightarrow{p}M
\end{align*}
It is straightforward to show that these two definitions are equivalent.
\end{remark}

\begin{example}
\label{example:cAlg-op-as-slice-tangent-category}
In~\cite[Section~4.1]{cruttwell:algebraic-geometry}, the first author and Lemay showed that the tangent category $(\cAlg^\op_R,\TT)$ of affine schemes of $R$ can also be characterized as the slice tangent category of $(\cRing^\op,\TT)$ over the ring $R$. Indeed, unital and commutative algebras over a ring $R$ are equivalently characterized as morphisms $R\to M$ of rings. Since the display maps in $(\cAlg^\op_R,\TT)$ are all maps (Example \ref{ex:display_in_AG}), this is the same as the display slice tangent category. 
\end{example}

One can see that the display slice tangent category $\Dsply(\X,\TT;M)$ of any (not-necessarily cartesian) tangent category over a given object is always a cartesian tangent category. Indeed, since the objects of $\Dsply(\X,\TT;M)$ are the tangent display maps $q\colon E\to M$ of $(\X,\TT)$ with codomain $M$, given two such tangent display maps $q\colon E\to M$ and $q'\colon E'\to M$, the pullback $q\times_Mq'\colon E\times_ME'\to M$ in $\X$ of $q$ along $q'$ always exists.
\par However, $q\times_Mq'$ is the cartesian product of $q$ with $q'$ in $(\X,\TT)/M$. By induction, all finite products (the terminal objects will be the tangent display map $\id_M\colon M\to M$) exist in $(\X,\TT)/M$. Moreover, the tangent bundle functor $\T\^M$ of $\Dsply(\X,\TT;M)$ preserves products in $\Dsply(\X,\TT;M)$ since $\T$ in $\X$ preserves pullbacks between tangent display maps.

\begin{lemma}
\label{lemma:slice-tangent-category-cartesian}
The display slice tangent category $\Dsply(\X,\TT;M)$ of a tangent category $(\X,\TT)$ over a given object $M$ of $\X$ is a cartesian tangent category.
\end{lemma}

Our goal for the rest of this section is to provide a universal property for the display slice construction. We begin by introducing the category of tangent pairs.

\begin{definition}
\label{definition:tangent-pair}
A \textbf{tangent pair} consists of a pair $((\X,\TT);M)$ formed by a tangent category $(\X,\TT)$ and an object $M$ of $\X$. Moreover, a morphism $((\X,\TT);M)\to((\X',\TT');M')$ of tangent pairs consists of a pair $((F,\alpha);\varphi)$ formed by a lax tangent morphism $(F,\alpha)\colon(\X,\TT)\to(\X',\TT')$ which preserves tangent display maps over $M$ i.e., for which every tangent display map $q\colon E\to M$ is sent to a tangent display map $Fq\colon FE\to FM$, together with an isomorphism $\varphi\colon FM\to M'$ of $\X'$.
\end{definition}

The composition of two morphisms $((F,\alpha);\varphi)\colon((\X,\TT);M)\to((\X',\TT');M')$ and $((G,\beta);\psi)\colon((\X',\TT');M')\to((\X'',\TT'');M'')$ of tangent pairs is defined as the tangent morphism $(G,\beta)\o(F,\alpha)=(G\o F,G\alpha\beta_F)$ together with the isomorphism:
\begin{align*}
&GFM\xrightarrow{G\varphi}GM'\xrightarrow{\psi}M''
\end{align*}
Notice, in particular, that since $G$ preserves tangent display maps and tangent display maps are closed under composition, this defines a morphism of tangent pairs. Therefore, tangent pairs together with their morphisms form a category, which we denote by $\TngPair$.\newline
\par Next, we introduce the pseudofunctor $\Term\colon\cTngCat\to\TngPair$ (where $\cTngCat$ is the 2-category of cartesian tangent categories). 
First, fix some terminal object $\ast$ of $\X$. Then define $\Term$ to send a cartesian tangent category $(\X,\TT)$ to the tangent pair $((\X,\TT);\ast)$.
\par Notice that, since pullbacks over a terminal object are precisely cartesian products, tangent display maps over a terminal object consist of those objects $E$ of $\X$ for which the cartesian product $E'\times E$ exists for any other object $E'$ and for which this product is preserved by $\T$, i.e., $\T(E'\times E)\cong\T E'\times\T E$. Therefore, for a cartesian tangent category $(\X,\TT)$, tangent display maps over a terminal object are all the objects.
\par This observation implies that $\Term$ is a well-defined functor. Indeed, given a cartesian tangent morphism $(F,\alpha)\colon(\X,\TT)\to(\X',\TT')$ let $\Term(F,\alpha)$ be the pair $((F,\alpha);!)$ formed by $(F,\alpha)$ (notice that since $F$ preserves the terminal object, $F$ sends tangent display maps of $(\X,\TT)$ on the terminal object $\ast$ of $\X$ to tangent display maps of $\X'$ over $F\ast\cong\ast'$) and by the isomorphism $!\colon F\ast\to\ast'$ (which is trivially a tangent display map since it is an isomorphism).\newline
\par On the other hand, we can also define the pseudofunctor $\Slice\colon\TngPair\to\cTngCat$ which sends a tangent pair $((\X,\TT);M)$ to the display slice tangent category $\Dsply(\X,\TT;M)$. To understand how $\Slice$ acts on morphisms of tangent pairs, we first need to show that we can lift such a morphism to the slice tangent categories.

\begin{remark}
\label{remark:pseudofunctoriality-slice-term}
$\Term$ and $\Slice$ are not strict functors but rather pseudofunctors, since terminal objects and slice tangent structures are only defined up to unique isomorphisms. Thus, the associators and unitors are defined by the induced unique isomorphisms.
\end{remark}

\begin{proposition}
\label{proposition:lifting-tangent-pair-morphisms-to-slice}
Consider two tangent pairs $((\X,\TT);M)$ and $((\X',\TT');M')$ and a morphism of tangent pairs $((F,\alpha);\varphi)\colon((\X,\TT);M)\to((\X',\TT');M')$. Let $q\colon E\to M$ be a tangent display map in $(\X,\TT)$ over $M$. Finally, consider the morphism $\theta_q\colon F\T\^ME\to{\T'}\^{M'}FE$ as the unique morphism which makes the following diagram commutes:
\begin{equation*}
\begin{tikzcd}
{F\T\^ME} &&&& {F\T E} \\
& {{\T'}\^{M'}FE} && {\T'FE} \\
&&& {\T'FM} \\
& {M'} && {\T'M'} \\
FM &&&& {F\T M}
\arrow["{F\iota_q}", from=1-1, to=1-5]
\arrow["{\theta_q}"{pos=0.7}, dashed, from=1-1, to=2-2]
\arrow["{Fq\^M}"'{pos=0.6}, from=1-1, to=5-1]
\arrow["\alpha"{description}, from=1-5, to=2-4]
\arrow["{F\T q}"{pos=0.6}, from=1-5, to=5-5]
\arrow["{\iota_{Fq}}", from=2-2, to=2-4]
\arrow["{(Fq\varphi)\^{M'}}", from=2-2, to=4-2]
\arrow["\lrcorner"{anchor=center, pos=0.125}, draw=none, from=2-2, to=4-4]
\arrow["{\T'Fq}"', from=2-4, to=3-4]
\arrow["{\T'\varphi}"', from=3-4, to=4-4]
\arrow["z"', from=4-2, to=4-4]
\arrow["\varphi"{description}, from=5-1, to=4-2]
\arrow["Fz"', from=5-1, to=5-5]
\arrow["\alpha"{description}, from=5-5, to=3-4]
\end{tikzcd}
\end{equation*}
Therefore, the functor:
\begin{align*}
&F\colon\X/M\to{\X'}/M'\\
&F(q\colon E\to M)\mapsto(FE\xrightarrow{Fq}FM\xrightarrow{\varphi}M')\\
&F(g\colon(q\colon E\to M)\to(q'\colon E'\to M))\mapsto(Fg\colon(Fq\varphi)\to(Fq'\varphi))
\end{align*}
extends to a lax tangent morphism:
\begin{align*}
&(F,\alpha)/\varphi\colon(\X,\TT)/M\to(\X',\TT')/M'
\end{align*}
whose distributive law is defined by the natural transformation $\theta_q\colon F\T\^Mq\to{\T'}\^{M'}Fq$.
\begin{proof}
For starters, let us prove the compatibility between $\theta$ and the projections:
\begin{equation*}
\begin{tikzcd}
{F\T\^Mq} & {{\T'}\^{M'}Fq} \\
Fq
\arrow["{Fp\^M}"', from=1-1, to=2-1]
\arrow["{p\^{M'}_F}", from=1-2, to=2-1]
\arrow["\theta", from=1-1, to=1-2]
\end{tikzcd}
\end{equation*}
which corresponds to the diagram:
\begin{equation*}
\begin{tikzcd}
{F\T\^ME} & {{\T'}\^{M'}FE} \\
{F\T E} & {\T' FE} \\
FE & FE
\arrow[""{name=0, anchor=center, inner sep=0}, "\theta", from=1-1, to=1-2]
\arrow["{\iota_F}", from=1-2, to=2-2]
\arrow["{p_F}", from=2-2, to=3-2]
\arrow["F\iota"', from=1-1, to=2-1]
\arrow["Fp"', from=2-1, to=3-1]
\arrow[""{name=1, anchor=center, inner sep=0}, Rightarrow, no head, from=3-1, to=3-2]
\arrow[""{name=2, anchor=center, inner sep=0}, "\alpha"{description}, from=2-1, to=2-2]
\end{tikzcd}
\end{equation*}
Let us consider the compatibility diagram between $\theta$ and the zero morphisms:
\begin{equation*}
\begin{tikzcd}
{F\T\^Mq} & {{\T'}\^{M'}Fq} \\
Fq
\arrow["\theta", from=1-1, to=1-2]
\arrow["{Fz\^M}", from=2-1, to=1-1]
\arrow["{z\^{M'}_F}"', from=2-1, to=1-2]
\end{tikzcd}
\end{equation*}
To show that, first, consider the diagram:
\begin{equation*}
\begin{tikzcd}
{F\T\^ME} &&& {{\T'}\^{M'}FE} \\
& {F\T E} & {\T' FE} \\
FE &&& FE
\arrow[""{name=0, anchor=center, inner sep=0}, "\theta", from=1-1, to=1-4]
\arrow["{\iota_F}"', from=1-4, to=2-3]
\arrow[""{name=1, anchor=center, inner sep=0}, "{Fz\^M}", from=3-1, to=1-1]
\arrow["F\iota", from=1-1, to=2-2]
\arrow[""{name=2, anchor=center, inner sep=0}, "\alpha"', from=2-2, to=2-3]
\arrow["Fz"', from=3-1, to=2-2]
\arrow[""{name=3, anchor=center, inner sep=0}, Rightarrow, no head, from=3-1, to=3-4]
\arrow["{z_F}", from=3-4, to=2-3]
\arrow[""{name=4, anchor=center, inner sep=0}, "{z\^{M'}_F}"', from=3-4, to=1-4]
\end{tikzcd}
\end{equation*}
Thus $Fz\^M\theta \iota_F=z\^{M'}_F\iota_F$ and from the universality of $\iota_F$ we have that $Fz\^M\theta=z\^{M'}_F$, as expected. The next step is to prove the compatibility with the sum morphism:
\begin{equation*}
\begin{tikzcd}
{F\T\^M_2q} & {{\T'}\^{M'}_2Fq} \\
{F\T\^Mq} & {{\T'}\^{M'}Fq}
\arrow["\theta\times\theta", from=1-1, to=1-2]
\arrow["{Fs\^M}"', from=1-1, to=2-1]
\arrow["{s\^M_F}", from=1-2, to=2-2]
\arrow["\theta"', from=2-1, to=2-2]
\end{tikzcd}
\end{equation*}
Thus, consider the following diagram:
\begin{equation*}
\begin{tikzcd}
{F\T\^M_2E} &&& {{\T'}\^{M'}_2FE} \\
& {F\T_2E} & {\T'_2FE} \\
& {F\T E} & {\T' FE} \\
{F\T\^ME} &&& {{\T'}\^{M'}FE}
\arrow[""{name=0, anchor=center, inner sep=0}, "\theta"', from=4-1, to=4-4]
\arrow["{\iota_F}", from=4-4, to=3-3]
\arrow[""{name=1, anchor=center, inner sep=0}, "{Fs\^M}"', from=1-1, to=4-1]
\arrow["F\iota"', from=4-1, to=3-2]
\arrow[""{name=2, anchor=center, inner sep=0}, "\alpha", from=3-2, to=3-3]
\arrow[""{name=3, anchor=center, inner sep=0}, "{s\^{M'}_F}", from=1-4, to=4-4]
\arrow["{\iota_F\times \iota_F}"', from=1-4, to=2-3]
\arrow[""{name=4, anchor=center, inner sep=0}, "{s_F}", from=2-3, to=3-3]
\arrow["{F\iota\times F\iota}", from=1-1, to=2-2]
\arrow[""{name=5, anchor=center, inner sep=0}, "Fs"', from=2-2, to=3-2]
\arrow[""{name=6, anchor=center, inner sep=0}, "\alpha\times\alpha", from=2-2, to=2-3]
\arrow[""{name=7, anchor=center, inner sep=0}, "\theta\times\theta", from=1-1, to=1-4]
\end{tikzcd}
\end{equation*}
Thus, $Fs\^M\theta \iota_F=(\theta\times\theta)s\^{M'}_F\iota_F$ and from the universality of $\iota_F$ we conclude that $Fs\^M\theta=(\theta\times\theta)s\^{M'}$, as expected. Let us prove the compatibility with the lift:
\begin{equation*}
\begin{tikzcd}
{F\T\^Mq} && {{\T'}\^{M'}Fq} \\
{F(\T\^M)^2q} & {{\T'}\^{M'}F\T\^Mq} & {({\T'}\^{M'})^2Fq}
\arrow["\theta", from=1-1, to=1-3]
\arrow["{Fl\^M}"', from=1-1, to=2-1]
\arrow["{l\^{M'}_F}", from=1-3, to=2-3]
\arrow["{\theta_{\T\^M}}"', from=2-1, to=2-2]
\arrow["{{\T'}\^{M'}\theta}"', from=2-2, to=2-3]
\end{tikzcd}
\end{equation*}
As before, consider the following diagram:
\begin{equation*}
\adjustbox{width=\linewidth,center}{
\begin{tikzcd}
{F\T\^ME} &&&& {{\T'}\^{M'}FE} \\
& {F\T E} && {\T' FE} \\
& {F\T^2E} & {\T' F\T E} & {\T'^2FE} \\
& {F\T\^M\T E} & {{\T'}\^{M'}F\T E} & {{\T'}\^{M'}\T' FE} \\
{F(\T\^M)^2E} && {{\T'}\^{M'}F\T\^ME} && {({\T'}\^{M'})^2FE}
\arrow[""{name=0, anchor=center, inner sep=0}, "\theta", from=1-1, to=1-5]
\arrow[""{name=1, anchor=center, inner sep=0}, "{l\^{M'}_F}", from=1-5, to=5-5]
\arrow["{\iota_{\T' F}}"', from=4-4, to=3-4]
\arrow[""{name=2, anchor=center, inner sep=0}, "{Fl\^M}"', from=1-1, to=5-1]
\arrow["{\iota_F}"', from=1-5, to=2-4]
\arrow["{l_F}", from=2-4, to=3-4]
\arrow["F\iota"', from=1-1, to=2-2]
\arrow[""{name=3, anchor=center, inner sep=0}, "\alpha", from=2-2, to=2-4]
\arrow["Fl"', from=2-2, to=3-2]
\arrow["{F\T\^M\iota}"{pos=0.3}, from=5-1, to=4-2]
\arrow["{F\iota_\T}", from=4-2, to=3-2]
\arrow["{\alpha_\T}", from=3-2, to=3-3]
\arrow["\T'\alpha", from=3-3, to=3-4]
\arrow[""{name=4, anchor=center, inner sep=0}, "{\theta_{\T\^M}}"', from=5-1, to=5-3]
\arrow[""{name=5, anchor=center, inner sep=0}, "{{\T'}\^{M'}\theta}"', from=5-3, to=5-5]
\arrow["{{\T'}\^{M'}\iota}"'{pos=0.3}, from=5-5, to=4-4]
\arrow["{{\T'}\^{M'}F\iota}"{description}, from=5-3, to=4-3]
\arrow["{{\T'}\^{M'}\alpha}", from=4-3, to=4-4]
\arrow["{\theta_\T}", from=4-2, to=4-3]
\arrow["{\iota_{F\T}}"{description}, from=4-3, to=3-3]
\end{tikzcd}
}
\end{equation*}
Therefore, $\theta l\^{M'}_F{\T'}\^{M'}\iota\iota_{\T' F}=Fl\^M\theta_{\T\^M}{\T'}\^{M'}\theta{\T'}\^{M'}\iota\iota_{\T' F}$. By the universality of ${\T'}\^{M'}\iota\iota_{\T' F}$ we conclude that $\theta l\^{M'}_F=Fl\^M\theta_{\T\^M}{\T'}\^{M'}\theta$, as expected. Finally, let us prove the compatibility with the canonical flip:
\begin{equation*}
\begin{tikzcd}
{F(\T\^M)^2q} & {{\T'}\^{M'}F\T\^Mq} & {({\T'}\^{M'})^2Fq} \\
{F(\T\^M)^2q} & {{\T'}\^{M'}F\T\^Mq} & {({\T'}\^{M'})^2Fq}
\arrow["{Fc\^M}"', from=1-1, to=2-1]
\arrow["{c\^{M'}_F}", from=1-3, to=2-3]
\arrow["{\theta_{\T\^M}}"', from=2-1, to=2-2]
\arrow["{{\T'}\^{M'}\theta}"', from=2-2, to=2-3]
\arrow["{\theta_{\T\^M}}", from=1-1, to=1-2]
\arrow["{{\T'}\^{M'}\theta}", from=1-2, to=1-3]
\end{tikzcd}
\end{equation*}
Thus:
\begin{equation*}
\adjustbox{width=\linewidth,center}{
\begin{tikzcd}
{F(\T\^M)^2E} && {{\T'}\^{M'}F\T\^ME} && {({\T'}\^{M'})^2FE} \\
& {F\T\^M\T E} & {{\T'}\^{M'}F\T E} & {{\T'}\^{M'}\T' FE} \\
& {F\T E} & {\T' F\T E} & {\T'^2FE} \\
& {F\T^2E} & {\T' F\T E} & {\T'^2FE} \\
& {F\T\^M\T E} & {{\T'}\^{M'}F\T E} & {{\T'}\^{M'}\T' FE} \\
{F(\T\^M)^2E} && {{\T'}\^{M'}F\T\^ME} && {({\T'}\^{M'})^2FE}
\arrow["{\iota_{\T' F}}"', from=5-4, to=4-4]
\arrow[""{name=0, anchor=center, inner sep=0}, "{Fc\^M}"', from=1-1, to=6-1]
\arrow["{c_F}", from=3-4, to=4-4]
\arrow["Fc"', from=3-2, to=4-2]
\arrow["{F\T\^M\iota}"{pos=0.3}, from=6-1, to=5-2]
\arrow["{F\iota_\T}", from=5-2, to=4-2]
\arrow["{\alpha_\T}", from=4-2, to=4-3]
\arrow["\T'\alpha", from=4-3, to=4-4]
\arrow[""{name=1, anchor=center, inner sep=0}, "{\theta_{\T\^M}}"', from=6-1, to=6-3]
\arrow[""{name=2, anchor=center, inner sep=0}, "{{\T'}\^{M'}\theta}"', from=6-3, to=6-5]
\arrow["{{\T'}\^{M'}\iota}"'{pos=0.3}, from=6-5, to=5-4]
\arrow["{{\T'}\^{M'}F\iota}"{description}, from=6-3, to=5-3]
\arrow["{{\T'}\^{M'}\alpha}", from=5-3, to=5-4]
\arrow["{\theta_\T}", from=5-2, to=5-3]
\arrow["{\iota_{F\T}}"{description}, from=5-3, to=4-3]
\arrow["\Nat"{description}, draw=none, from=4-3, to=5-4]
\arrow["{(\alpha,\theta;\iota)}"{description}, draw=none, from=4-3, to=5-2]
\arrow[""{name=3, anchor=center, inner sep=0}, "{c\^{M'}_F}", from=1-5, to=6-5]
\arrow["{{\T'}\^{M'}\iota}"', from=1-5, to=2-4]
\arrow["{\iota_{\T' F}}", from=2-4, to=3-4]
\arrow["{F\T\^M\iota}"', from=1-1, to=2-2]
\arrow["{F\iota_\T}"', from=2-2, to=3-2]
\arrow["{\iota_{F\T}}"{description}, from=2-3, to=3-3]
\arrow["{\alpha_\T}"', from=3-2, to=3-3]
\arrow["\T'\alpha"', from=3-3, to=3-4]
\arrow["{\theta_\T}"', from=2-2, to=2-3]
\arrow["{{\T'}\^{M'}\alpha}"', from=2-3, to=2-4]
\arrow["{{\T'}\^{M'}F\iota}"{description}, from=1-3, to=2-3]
\arrow[""{name=4, anchor=center, inner sep=0}, "{\theta_{\T\^M}}", from=1-1, to=1-3]
\arrow[""{name=5, anchor=center, inner sep=0}, "{{\T'}\^{M'}\theta}", from=1-3, to=1-5]
\end{tikzcd}
}
\end{equation*}
This proves that $\theta_{\T\^M}{\T'}\^{M'}\theta c\^{M'}_F{\T'}\^{M'}\iota\iota_{\T' F}=Fc\^M\theta_{\T\^M}{\T'}\^{M'}\theta{\T'}\^{M'}\iota\iota_{\T' F}$. Finally, using the universality of ${\T'}\^{M'}\iota\iota_{\T' F}$ we conclude that $\theta_{\T\^M}{\T'}\^{M'}\theta c\^{M'}_F=Fc\^M\theta_{\T\^M}{\T'}\^{M'}\theta$, as expected.
\end{proof}
\end{proposition}

Proposition~\ref{proposition:lifting-tangent-pair-morphisms-to-slice} allows one to lift morphisms of tangent pairs to the corresponding slice tangent categories. We can define $\Slice$ to be the pseudofunctor which sends a morphism $((F,\alpha);\varphi)\colon((\X,\TT);M)\to((\X',\TT');M')$ to the lax tangent morphism $(F,\alpha)/\varphi$. Notice also that since $F$ preserves tangent display maps over $M$, it also preserves the cartesian products between tangent display maps over $M$.

\begin{remark}
\label{remark:varphi-isomorphism}
To define morphisms of tangent pairs one could have simply asked $\varphi$ to be a tangent display map. However, in order for $(F,\alpha)/\varphi$ to preserve cartesian products we needed $\varphi$ to be an isomorphism.
\end{remark}

\begin{definition}
\label{definition:cartesian-morphisms-of-tangent-pairs}
A morphism 
$((F,\alpha);\varphi)\colon((\X,\TT);M)\to((\X',\TT');M')$ of tangent pairs is \textbf{cartesian} if the following diagrams:
\begin{equation*}
\begin{tikzcd}
{F\T E} & {\T' FE} \\
{F\T M} & {\T' FM}
\arrow["{F\T q}"', from=1-1, to=2-1]
\arrow["{\T' Fq}", from=1-2, to=2-2]
\arrow["\alpha"', from=2-1, to=2-2]
\arrow["\alpha", from=1-1, to=1-2]
\end{tikzcd}\hfill
\begin{tikzcd}
{F\T\^ME} & {F\T E} \\
FM & {F\T M}
\arrow["{F\T q}", from=1-2, to=2-2]
\arrow["Fz"', from=2-1, to=2-2]
\arrow["{F\iota_q}", from=1-1, to=1-2]
\arrow["{Fq^\*}"', from=1-1, to=2-1]
\end{tikzcd}
\end{equation*}
are pullback diagrams, for every tangent display map $q\colon E\to M$ of $(\X,\TT)$ over $M$.
\end{definition}

\begin{remark}
\label{remark:display-bundle-preserving-not-cartesian}
Even if the functor $F$ underlying a morphism of tangent pairs preserves tangent display maps over the given object $M$ of the pair, it is not guaranteed that $F$ preserves tangent display maps over $\T M$. This is the reason why in Definition~\ref{definition:cartesian-morphisms-of-tangent-pairs} we require the right diagram to be a pullback.
\end{remark}

\begin{lemma}
\label{lemma:cartesian-morphisms-become-strong}
A cartesian morphism of tangent pairs $((F,\alpha);\varphi)\colon((\X,\TT);M)\to((\X',\TT');M')$ lifts to a strong tangent morphism between the slice tangent categories. Concretely, this means that the natural transformation $\theta_q\colon F\T\^Mq\to{\T'}\^{M'}Fq$ defined in Proposition~\ref{proposition:lifting-tangent-pair-morphisms-to-slice} is invertible.
\begin{proof}
Consider the following diagram:
\begin{equation*}
\begin{tikzcd}
{F\T\^ME} & {F\T E} & {\T' FE} \\
& {F\T M} \\
FM && {\T' FM} \\
M' && {\T' M'}
\arrow["Fv", from=1-1, to=1-2]
\arrow["\alpha", from=1-2, to=1-3]
\arrow["{F\T\^Mq}"', from=1-1, to=3-1]
\arrow["Fz", from=3-1, to=2-2]
\arrow["{F\T q}"', from=1-2, to=2-2]
\arrow["\alpha", from=2-2, to=3-3]
\arrow[""{name=0, anchor=center, inner sep=0}, "{\T' Ff}", from=1-3, to=3-3]
\arrow["{z_F}"', from=3-1, to=3-3]
\arrow["\varphi"', from=3-1, to=4-1]
\arrow[""{name=1, anchor=center, inner sep=0}, "z"', from=4-1, to=4-3]
\arrow["\T'\varphi", from=3-3, to=4-3]
\arrow["\lrcorner"{anchor=center, pos=0.125}, draw=none, from=1-1, to=2-2]
\arrow["\lrcorner"{anchor=center, pos=0.125}, draw=none, from=1-2, to=0]
\arrow["\lrcorner"{anchor=center, pos=0.125}, draw=none, from=3-1, to=1]
\end{tikzcd}
\end{equation*}
where we used that $Fz\alpha=z_F$. Since $((F,\alpha);\varphi)$ is cartesian, such a diagram is formed by pullback diagrams, thus it is itself a pullback diagram. (Notice that the bottom square diagram is a pullback because $\varphi$ is an isomorphism.) On the other hand, by definition, $\theta$ is defined by the diagram:
\begin{equation*}
\adjustbox{scale=.7,center}{
\begin{tikzcd}
{{\T'}\^{M'}FE} \\
&&& {\T' FE} \\
M' && {F\T E} & {\T' FM} \\
& {F\T\^ME} && {\T' M'} \\
& FM \\
& M'
\arrow["Fv", from=4-2, to=3-3]
\arrow["\alpha", from=3-3, to=2-4]
\arrow["{F\T\^Mq}", from=4-2, to=5-2]
\arrow["{\T' Fq}", from=2-4, to=3-4]
\arrow["\varphi", from=5-2, to=6-2]
\arrow["z"', from=6-2, to=4-4]
\arrow["\T'\varphi", from=3-4, to=4-4]
\arrow["{\iota_F}", from=1-1, to=2-4]
\arrow["{{\T'}\^{M'}(Fq\varphi)}"', from=1-1, to=3-1]
\arrow["z"'{pos=0.7}, from=3-1, to=4-4]
\arrow["\theta"', dashed, from=4-2, to=1-1]
\arrow[Rightarrow, no head, from=6-2, to=3-1]
\end{tikzcd}
}
\end{equation*}
However, the top and the right rectangular sides of this triangular diagram are pullbacks, so $\theta$ must be an isomorphism.
\end{proof}
\end{lemma}

We can now characterize the operation which takes a tangent pair to its slice tangent category as an adjunction between pseudofunctors.

\begin{theorem}
\label{theorem:adjunction-Term-Slice}
The pseudofunctors $\Term\colon\cTngCat\leftrightarrows\TngPair\colon\Slice$ form an adjunction whose left adjoint is $\Term$, the right adjoint is $\Slice$, the unit $(\U,\eta)\colon(\X,\TT)\to\Slice(\Term(\X,\TT))=(\X,\TT)/\ast$, as a cartesian tangent morphism between cartesian tangent categories, is the isomorphism:
\begin{align*}
&\U\colon\X\to\X/\ast\\
&\U(M)\mapsto(!\colon M\to\*)\\
&\U(f\colon M\to N)\mapsto(f(!\colon M\to\*)\to(!\colon N\to\*))\\
&\eta\colon(\U(\T M))=(!\colon\T M\to\*)\xrightarrow{\id_{\T M}}(!\colon\T M\to\*)=\T(\U(M))
\end{align*}
and the counit $((C,\epsilon);\varphi)\colon\Term(\Slice((\X,\TT);M))=((\X,\TT)/M;\id_M)\to((\X,\TT);M)$ is the morphism of tangent pairs:
\begin{align*}
&C\colon(\X,\TT)/M\to(\X,\TT)\\
&C(q\colon E\to M)\mapsto E\\
&C(g\colon(q\colon E\to M)\to(q'\colon E'\to M))\mapsto(g\colon E\to E')\\
&\epsilon\colon C(\T\^M(q\colon E\to M))=\T\^ME\xrightarrow{\iota_q}\T E=\T(C(q\colon E\to M))\\
&\varphi\colon C(\id_M\colon M\to M)=M\xrightarrow{\id_M}M
\end{align*}
\begin{proof}
Let us start by noticing that the unit and the counit are well-defined morphisms. The underlying functor $\U$ of the unit is clearly cartesian, so $(\U,\eta)$ is well-defined. Let us focus on the counit. A tangent display map in $(\X,\TT)/M$ over $\id_M\colon M\to M$ consists of an object $q\colon E\to M$ of $(\X,\TT)/M$, i.e., a tangent display map of $(\X,\TT)$ over $M$, together with a morphism $q'\colon E'\to M$ for which $q'\id_M=q$. This implies that tangent display maps of $(\X,\TT)/M$ over $\id_M$ are also tangent display maps of $(\X,\TT)$ over $M$. So, the underlying functor $C$ of the counit sends tangent display maps to tangent display maps.
\par The next step is to show that the unit $(\U,\eta)$ and the counit $((C,\epsilon);\varphi)$ satisfy the triangle identities. Let us start by considering the following diagram:
\begin{equation*}
\begin{tikzcd}
{\Term(\X,\TT)} & {\Term(\Slice(\Term(\X,\TT)))} \\
& {\Term(\X,\TT)}
\arrow["{\Term(\U,\eta)}", from=1-1, to=1-2]
\arrow["{((C,\epsilon);\varphi)_\Term}", from=1-2, to=2-2]
\arrow[Rightarrow, no head, from=1-1, to=2-2]
\end{tikzcd}
\end{equation*}
for a cartesian tangent category $(\X,\TT)$. It is straightforward to see that the underlying tangent morphisms $(C,\epsilon)$ and $(\U,\eta)$ of $((C,\epsilon);\varphi)_{\Term}$ and $\Term(\U,\eta)$ define an equivalence between $(\X,\TT)$ and $(\X,\TT)/\*$ and that, by the universality of the terminal object, the composition of the comparison morphisms $\varphi=\id_\*$ and $!\colon \U\*\to\*$ is the identity over the terminal object. Similarly, by considering the diagram:
\begin{equation*}
\begin{tikzcd}
{\Slice((\X,\TT);M)} & {\Slice(\Term(\Slice((\X,\TT);M)))} \\
& {\Slice((\X,\TT);M)}
\arrow["{(\U,\eta)_\Slice}", from=1-1, to=1-2]
\arrow["{\Slice((C,\epsilon);\varphi)}", from=1-2, to=2-2]
\arrow[Rightarrow, no head, from=1-1, to=2-2]
\end{tikzcd}
\end{equation*}
for a tangent pair $((\X,\TT);M)$, it is straightforward to show the underlying tangent morphisms of $\Slice((C,\epsilon);\varphi)$ and $(\U,\eta)_{\Slice}$ define an equivalence between $(\X,\TT)/M$ and $((\X,\TT)/M)/\id_M$ and that the composition of the comparison morphisms gives the identity. Finally, notice that the unit is always an isomorphism.
\end{proof}
\end{theorem}

\begin{remark}
\label{remark:paper-version-on-slicing}
As mentioned earlier, in~\cite{lanfranchi:differential-bundles-operadic-affine-schemes}, the second author used a different approach to defining the slice tangent category. Instead of considering only tangent display maps as objects of the slice tangent category, all morphisms with a fixed codomain were introduced. However, since the existence of tangent pullbacks along these morphisms is required in order to define the slice tangent structure, only so-called \textit{sliceable objects} were considered. We suggest the interested reader to consult the original paper for details. This discrepancy in the definition of the slice tangent category in the original paper results in a different adjunction. The adjunction $\Term\dashv\Slice$ between tangent pairs and cartesian tangent categories was replaced by an adjunction between tangent pairs and tangent categories with a terminal object.
\end{remark}

We end this section by considering tangent display maps in the slice. While it is not clear how to characterize all such maps, we at least have the following result. 

\begin{proposition}
\label{proposition:display-sliceable}
Let $h\colon E\to P$ a morphism of $\X$ and suppose that $hg=f$ for $g\colon P\to M$ and $f\colon E\to M$. If $h$ is a tangent display map of $(\X,\TT)$, then it is also a tangent display map of $(\X,\TT)/M$.
\begin{proof}
Let $n$ be a positive integer and let $a\colon Q\to\T^{(M)n}P$. We want to show that the pullback of $\T^{(M)n}h\colon\T^{(M)n}E\to\T^{(M)n}P$ along $a$ exists and that is a tangent pullback. Let us start by showing the existence. Since $h$ is a $\T$-display map, the pullback:
\begin{equation*}
\begin{tikzcd}
Z && {\T^nE} \\
Q & {\T^{(M)n}P} & {\T^nP}
\arrow[from=1-1, to=1-3]
\arrow[from=1-1, to=2-1]
\arrow["\lrcorner"{anchor=center, pos=0.125}, draw=none, from=1-1, to=2-2]
\arrow["{\T^nh}", from=1-3, to=2-3]
\arrow["a"', from=2-1, to=2-2]
\arrow[from=2-2, to=2-3]
\end{tikzcd}
\end{equation*}
exists and it is also preserved by every $\T^m$. However, we also have the existence of a dashed arrow as follows:
\begin{equation*}
\begin{tikzcd}
Z && {\T^nE} \\
Q & {\T^{(M)n}E} & {\T^nE} \\
{\T^{(M)n}P} & M & {\T^nM}
\arrow[from=1-1, to=1-3]
\arrow[from=1-1, to=2-1]
\arrow[dashed, from=1-1, to=2-2]
\arrow[Rightarrow, no head, from=1-3, to=2-3]
\arrow["a"', from=2-1, to=3-1]
\arrow[from=2-2, to=2-3]
\arrow[from=2-2, to=3-2]
\arrow["\lrcorner"{anchor=center, pos=0.125}, draw=none, from=2-2, to=3-3]
\arrow["{\T^nf}", from=2-3, to=3-3]
\arrow[from=3-1, to=3-2]
\arrow["{z^n}"', from=3-2, to=3-3]
\end{tikzcd}
\end{equation*}
Thus in turn gives a commutative diagram:
\begin{equation*}
\begin{tikzcd}
Z & {\T^{(M)n}E} & {\T^nE} \\
Q & {\T^{(M)n}P} & {\T^nP}
\arrow[from=1-1, to=1-2]
\arrow[bend left, from=1-1, to=1-3]
\arrow[from=1-1, to=2-1]
\arrow[from=1-2, to=1-3]
\arrow["{\T^{(M)n}h}"{description}, from=1-2, to=2-2]
\arrow["{\T^nh}", from=1-3, to=2-3]
\arrow["a"', from=2-1, to=2-2]
\arrow[from=2-2, to=2-3]
\end{tikzcd}
\end{equation*}
Since the outer and the right squares are pullbacks, by Lemma~\ref{lemma:pullback-lemma}, so is the left one. This shows the existence of the desired pullback. Now, we need to show that this pullback is preserved by every $\T^{(M)m}$. First, notice that the pullback is preserved by every $\T^m$. This is a direct consequence of $h$ being a tangent display map. So, we have that:
\begin{equation*}
\begin{tikzcd}
{\T^mZ} & {\T^m\T^{(M)n}E} \\
{\T^mQ} & {\T^{(M)n}P}
\arrow[from=1-1, to=1-2]
\arrow[from=1-1, to=2-1]
\arrow["\lrcorner"{anchor=center, pos=0.125}, draw=none, from=1-1, to=2-2]
\arrow["{\T^m\T^{(M)n}h}", from=1-2, to=2-2]
\arrow["{\T^ma}"', from=2-1, to=2-2]
\end{tikzcd}
\end{equation*}
is a pullback diagram. Consider $\alpha\colon X\to\T^{(M)m}Q$ and $\beta\colon X\to\T^{(M)m+n}E$ such that:
\begin{equation*}
\begin{tikzcd}
X && {\T^mZ} && {\T^m\T^{(M)n}E} \\
& {\T^{(M)m}Z} && {\T^{(M)m+n}E} \\
&& {\T^mQ} && {\T^{(M)n}P} \\
& {\T^{(M)m}Q} && {\T^{(M)m+n}P}
\arrow[bend left, dashed, from=1-1, to=1-3]
\arrow["\beta"', bend left, from=1-1, to=2-4]
\arrow["\alpha"', bend right, from=1-1, to=4-2]
\arrow[from=1-3, to=1-5]
\arrow[from=1-3, to=3-3]
\arrow["\lrcorner"{anchor=center, pos=0.125}, draw=none, from=1-3, to=3-5]
\arrow["{\T^m\T^{(M)n}h}", from=1-5, to=3-5]
\arrow[from=2-2, to=1-3]
\arrow[from=2-2, to=2-4]
\arrow[from=2-2, to=4-2]
\arrow[from=2-4, to=1-5]
\arrow[from=2-4, to=4-4]
\arrow["{\T^ma}"', from=3-3, to=3-5]
\arrow[from=4-2, to=3-3]
\arrow[from=4-2, to=4-4]
\arrow[from=4-4, to=3-5]
\end{tikzcd}
\end{equation*}
However, we also have:
\begin{equation*}
\begin{tikzcd}
X && {\T^mZ} \\
& {\T^{(M)m}Z} & {\T^mZ} \\
{\T^{(M)m}Q} & M & {\T^mM}
\arrow[from=1-1, to=1-3]
\arrow[dashed, from=1-1, to=2-2]
\arrow["\alpha"', from=1-1, to=3-1]
\arrow[Rightarrow, no head, from=1-3, to=2-3]
\arrow[from=2-2, to=2-3]
\arrow[from=2-2, to=3-2]
\arrow["\lrcorner"{anchor=center, pos=0.125}, draw=none, from=2-2, to=3-3]
\arrow[from=2-3, to=3-3]
\arrow[from=3-1, to=3-2]
\arrow[from=3-2, to=3-3]
\end{tikzcd}
\end{equation*}
We leave the reader to show that the dashed morphism is the unique morphism which satisfies the desired equations.
\end{proof}
\end{proposition}

\subsection{Reverse tangent categories}
\label{subsection:reverse-tangent-category}
In the next two sections, we briefly consider how tangent display maps can be used to give alternate (simpler) formulations of a couple of other notions in tangent category theory.  As noted in the introduction, in a reverse tangent category, one needs a fibration of differential bundles on which to define an appropriate ``involution'' operation. In the paper which defines the notion, the authors thus ask for a system of differential bundles, that is a tangent display system of differential bundles, and ask for an involution operation on the corresponding fibration \cite[Defn. 23]{cruttwell:reverse-tangent-cats}  

However, this is additional structure which should not really be necessary, and is potentially a bit awkward to ``carry around'' as one works with reverse tangent categories. A much more natural choice is then to simply ask that a reverse tangent category be a tangent category which has an involution operation with respect to the tangent fibration of \emph{display} differential bundles.

\subsection{Linear completeness and curve objects}
\label{subsection:linear-completeness}
In \cite{cockett:differential-equations}, the authors define a ``linear curve object'' to be an object which gives unique solutions for differential equations, and also asks that solutions exist for linear differential equations. One of the issues with this definition, however, is the question of when solutions of linear equations carry over to related tangent categories. In particular, the authors want to determine whether or not a tangent category $\X$ has solutions for linear differential equations, then so does its tangent category of vector fields.

To do this, the authors focus on a particular class of ``endemic fibre products'', and look at differential bundles whose pullbacks lie in this class. Like in the previous section, this is then extra structure which has to be ``carried around''.

However, just as in the previous section, a natural choice is to simply consider the endemic fibre products to be the class of ``all pullbacks of tangent display maps'', and ask for linear completeness with respect to the resulting class of display differential bundles. All the results of \cite[Section 5.5]{cockett:differential-equations} then hold for this class, and again do not require any extra structure.

\subsection{Restriction tangent structures}
\label{subsection:restriction}
As a final application, in this section we look at the interaction between display maps in a tangent category and restriction structure.  In particular, we show that a certain class of tangent display maps allows one to build a canonical (tangent) restriction category out of any tangent category.

A restriction category consists of a category for which each morphism $f\colon A\to B$ comes equipped with an idempotent $\bar f\colon A\to A$, called the restriction idempotent of $f$. This structure abstracts the category of partial maps between sets in which the idempotent $\bar f$ of such a morphism is the partial map which coincides with the identity where $f$ is defined and is undefined elsewhere. In particular, restriction categories provide a simple equational theory for partial maps in an abstract category. For details, we suggest the reader to consult~\cite{cockett:restriction-categories-I}.

\par In~\cite{cockett:tangent-cats}, Cockett and the first author extended the theory of restriction categories to the realm of tangent categories, introducing the notion of a restriction tangent category (cf.~\cite[Definition~6.14]{cockett:tangent-cats}). In particular, a restriction tangent category is a restriction category equipped with a structure similar to a tangent structure, in which the tangent bundle functor preserves the restriction idempotents, the structural natural transformations are total natural transformations, i.e., natural transformations with trivial restriction idempotents, and for which the $n$-fold pullback of the projection along itself and the pullback of the universal property of the vertical lift are replaced with restriction pullbacks.

\par Cockett and Lack in~\cite{cockett:restriction-categories-I}, showed a $2$-equivalence between split restriction categories, i.e., restriction categories with splitting restriction idempotents, and categories equipped with a \textit{display system of monics}. Let us briefly recall this construction.

\begin{definition}
\label{definition:M-categories}
A \textbf{display system of monics} of a category $\X$ consists of a collection $\M$ of monomorphisms of $\X$ satisfying the following conditions:
\begin{itemize}
\item $\M$ forms a display system of $\X$;

\item $\M$ is closed under composition;

\item $\M$ contains all isomorphisms of $\X$. 
\end{itemize}
A category $\X$ equipped with a display system of monics is an \textbf{$\M$-category}.
\end{definition}

\begin{remark}
\label{remark:naming-display-system-monics}
In~\cite{cockett:restriction-categories-I}, a display system of monics is called a \textit{stable} system of monics.
\end{remark}

\begin{definition}
\label{definition:M-functors}
An \textbf{$\M$-functor} from an $\M$-category $(\X,\M)$ to an $\M$-category $(\X',\M')$ consists of a functor $F\colon\X\to\X'$ which sends each monic of $\M$ to a monic of $\M'$ and which preserves each pullback diagram of type:
\begin{equation*}
\begin{tikzcd}
D & B \\
C & A
\arrow[from=1-1, to=1-2]
\arrow[from=1-1, to=2-1]
\arrow["\lrcorner"{anchor=center, pos=0.125}, draw=none, from=1-1, to=2-2]
\arrow["m", from=1-2, to=2-2]
\arrow["f"', from=2-1, to=2-2]
\end{tikzcd}
\end{equation*}
where $m$ is a monic of $\M$.
\end{definition}

\begin{definition}
\label{definition:M-natural-transformation}
An \textbf{$\M$-natural transformation} from an $\M$-functor $F\colon(\X,\M)\to(\X',\M')$ to an $\M$-functor $G\colon(\X,\M)\to(\X',\M')$ consists of a natural transformation $\varphi\colon F\Rightarrow G$ for which, for every monic $m$ of $\M$, the naturality diagram:
\begin{equation}
\label{equation:M-natural-transformation}
\begin{tikzcd}
FA & GA \\
FB & GB
\arrow["\varphi", from=1-1, to=1-2]
\arrow["Fm"', from=1-1, to=2-1]
\arrow["Gm", from=1-2, to=2-2]
\arrow["\varphi"', from=2-1, to=2-2]
\end{tikzcd}
\end{equation}
is a pullback diagram.
\end{definition}

\begin{remark}
\label{remark:M-natural-transformations}
Note that the pullback diagram of Equation~\eqref{equation:M-natural-transformation} is preserved by every $\M$-functor from $(\X',\M')$, since $Gm$ is a monic in $\M'$.
\end{remark}

$\M$-categories, $\M$-functors, and $\M$-natural transformations form a $2$-category denoted by $\MCat$. Split restriction categories, restriction functors, i.e., functors which preserve the restriction idempotents, and total natural transformations also form a $2$-category denoted by $\sRestrCat$. Cockett and Lack proved that $\MCat$ and $\sRestrCat$ are $2$-equivalent.
\par In particular, each $\M$-category $(\X,\M)$ is sent to a split restriction category $\Par(\X,\M)$ whose objects are the same as the one of $\X$ and whose morphisms $f\colon A\nto B$ are classes of isomorphisms of spans:
\begin{equation*}
\begin{tikzcd}
& {\bar A} \\
A && B
\arrow["{m_f}"', from=1-2, to=2-1]
\arrow["{[f]}", from=1-2, to=2-3]
\end{tikzcd}
\end{equation*}
for which the leg $m_f$ is a monic of $\M$.

\par Our first goal of this section is to extend this result to split restriction tangent categories. The first step is to introduce the correct notion of a display system of monics in this context.

\begin{definition}
\label{definition:M-tangent-category}
A \textbf{tangent display system of monics} of a tangent category $(\X,\TT)$ consists of a collection $\M$ of morphisms of $(\X,\TT)$ satisfying the following conditions:
\begin{itemize}

\item Each element of $\M$ is a monomorphism;

\item Each element of $\M$ is an \'etale map of $(\X,\TT)$;

\item $\M$ forms a tangent display system of $(\X,\TT)$;

\item $\M$ is closed under composition;

\item $\M$ contains all isomorphisms of $(\X,\TT)$.
\end{itemize}
A tangent category equipped with a tangent display system of monics is called an \textbf{$\M$-tangent category}.
\end{definition}

\begin{remark}
\label{remark:tangent-display-system-monics}
We shall soon see why the \'etale requirement is important.  Also, it is worth noting that each element of a tangent display system of monics is automatically a \emph{display} \'etale map since every morphism of a tangent display system is a tangent display map by Theorem~\ref{theorem:T-display-maps-form-maximal-T-display-system}.
\end{remark}

\begin{definition}
\label{definition:M-tangent-functor}
A \textbf{lax $\M$-tangent morphism} from an $\M$-tangent category $(\X,\TT;\M)$ to an $\M$-tangent category $(\X',\TT';\M')$ consists of a lax tangent morphism $(F,\alpha)\colon(\X,\TT)\to(\X',\TT')$, i.e., a functor $F\colon\X\to\X'$ together with a distributive law $\alpha\colon F\o\T\Rightarrow\T'\o F$ compatible with the tangent structures, whose underlying functor $F$ is an $\M$-functor and whose distributive law $\alpha$ is an $\M$-natural transformation.
\end{definition}

\begin{definition}
\label{definition:M-tangent-natural-transformation}
An \textbf{$\M$-tangent natural transformation} from an $\M$-lax tangent morphism $(F,\alpha)\colon(\X,\TT;\M)\to(\X',\TT';\M')$ to an $\M$-lax tangent morphism $(G,\beta)\colon(\X,\TT;\M)\to(\X',\TT';\M')$ consists of a tangent natural transformation $\varphi\colon(F,\alpha)\Rightarrow(G,\beta)$, i.e., a natural transformation $\varphi\colon F\Rightarrow G$ compatible with the distributive laws $\alpha$ and $\beta$, for which, for every monic $m$ of $\M$, the naturality diagram:
\begin{equation*}
\begin{tikzcd}
FA & GA \\
FB & GB
\arrow["\varphi", from=1-1, to=1-2]
\arrow["Fm"', from=1-1, to=2-1]
\arrow["Gm", from=1-2, to=2-2]
\arrow["\varphi"', from=2-1, to=2-2]
\end{tikzcd}
\end{equation*}
is a tangent pullback diagram. 
\end{definition}

$\M$-tangent categories, lax $\M$-tangent functors, and $\M$-tangent natural transformations form a $2$-category denoted by $\MTngCat$.

\begin{lemma}
\label{lemma:etale-maps}
A morphism $q\colon E\to M$ of a tangent category $(\X,\TT)$ is \'etale if and only if the naturality diagrams of $q$ with all of the structural natural transformations of $\TT$ are tangent pullback diagrams.
\begin{proof}
By definition, a morphism $q\colon E\to M$ is an \'etale map if the naturality square of $q$ with the projection is a tangent pullback. Therefore, we only need to prove that if $q$ is \'etale, all the naturality squares of $q$ with each of the structural natural transformations is a tangent pullback diagram. Let us start with the zero morphism and consider the following diagram:
\begin{equation*}
\begin{tikzcd}
E & {\T E} & E \\
M & {\T M} & M
\arrow["z", from=1-1, to=1-2]
\arrow["q"', from=1-1, to=2-1]
\arrow["p", from=1-2, to=1-3]
\arrow["{\T q}"{description}, from=1-2, to=2-2]
\arrow["\lrcorner"{anchor=center, pos=0.125}, draw=none, from=1-2, to=2-3]
\arrow["q", from=1-3, to=2-3]
\arrow["z"', from=2-1, to=2-2]
\arrow["p"', from=2-2, to=2-3]
\end{tikzcd}
\end{equation*}
However, by assumption, the right square is a tangent pullback. Furthermore, since $zp=\id_\X$ also the outer square is a tangent pullback. By Lemma~\ref{lemma:pullback-lemma}, also the left square is a tangent pullback diagram. Let us now consider the diagrams:
\begin{equation*}
\begin{tikzcd}
{\T E} & {\T^2E} & {\T E} \\
{\T M} & {\T^2M} & {\T M}
\arrow["l", from=1-1, to=1-2]
\arrow["{\T q}"', from=1-1, to=2-1]
\arrow["{p_\T}", from=1-2, to=1-3]
\arrow["{\T^2q}"{description}, from=1-2, to=2-2]
\arrow["\lrcorner"{anchor=center, pos=0.125}, draw=none, from=1-2, to=2-3]
\arrow["{\T q}", from=1-3, to=2-3]
\arrow["l"', from=2-1, to=2-2]
\arrow["{p_\T}"', from=2-2, to=2-3]
\end{tikzcd}\hfill
\begin{tikzcd}
{\T^2E} & {\T^2E} & {\T E} \\
{\T^2M} & {\T^2M} & {\T M}
\arrow["c", from=1-1, to=1-2]
\arrow["{\T^2q}"', from=1-1, to=2-1]
\arrow["{p_\T}", from=1-2, to=1-3]
\arrow["{\T^2q}"{description}, from=1-2, to=2-2]
\arrow["\lrcorner"{anchor=center, pos=0.125}, draw=none, from=1-2, to=2-3]
\arrow["{\T q}", from=1-3, to=2-3]
\arrow["c"', from=2-1, to=2-2]
\arrow["{p_\T}"', from=2-2, to=2-3]
\end{tikzcd}
\end{equation*}
One can notice that the right squares are tangent pullbacks. Furthermore, $lp_\T=pz$ and $cp_\T=\T p$, thus, using that the naturality square of $q$ with $z$ is a tangent pullback, also the outer squares are tangent pullbacks. So, again by Lemma~\ref{lemma:pullback-lemma}, so are the left squares. Finally, let us prove it for the sum morphism. Let us consider the diagram:
\begin{equation}
\label{equation:sum-naturality}
\begin{tikzcd}
{\T_2E} & {\T E} & E \\
{\T_2M} & {\T M} & M
\arrow["s", from=1-1, to=1-2]
\arrow["{\T_2q}"', from=1-1, to=2-1]
\arrow["p", from=1-2, to=1-3]
\arrow["{\T q}"{description}, from=1-2, to=2-2]
\arrow["\lrcorner"{anchor=center, pos=0.125}, draw=none, from=1-2, to=2-3]
\arrow["q", from=1-3, to=2-3]
\arrow["s"', from=2-1, to=2-2]
\arrow["p"', from=2-2, to=2-3]
\end{tikzcd}
\end{equation}
Notice that $sp=\pi_1p=\pi_2p$. Consider two morphisms $f\colon X\to\T_2 M$ and $g\colon X\to E$ satisfying $f\pi_1p=g\T q$. We find a morphism $h\colon X\to\T E$ such that:
\begin{equation*}
\begin{tikzcd}
X \\
& {\T_2E} & {\T E} & E \\
& {\T_2M} & {\T M} & M
\arrow["h"{description}, bend left, dashed, from=1-1, to=2-3]
\arrow["g", bend left=30, from=1-1, to=2-4]
\arrow["f"', bend right=30, from=1-1, to=3-2]
\arrow["{\pi_1}", from=2-2, to=2-3]
\arrow["{\T_2q}"', from=2-2, to=3-2]
\arrow["p", from=2-3, to=2-4]
\arrow["{\T q}"{description}, from=2-3, to=3-3]
\arrow["\lrcorner"{anchor=center, pos=0.125}, draw=none, from=2-3, to=3-4]
\arrow["q", from=2-4, to=3-4]
\arrow["{\pi_1}"', from=3-2, to=3-3]
\arrow["p"', from=3-3, to=3-4]
\end{tikzcd}
\end{equation*}
Let:
\begin{align*}
&k\colon X\xrightarrow{h}\T E\xrightarrow{\Delta}\T_2E
\end{align*}
where $\Delta=\<\id_{\T E},\id_{\T E}\>$ is the diagonal map. It is immediate to see that:
\begin{align*}
&k\pi_1=h\Delta\pi_1=h
\end{align*}
Moreover, notice that, since $\pi_1p=\pi_2p$, we have that:
\begin{align*}
&f\pi_1=h\T q=f\pi_2
\end{align*}
Therefore:
\begin{align*}
&k\T_2q=h\Delta\T_2q=h\T q\Delta=f\pi_1\Delta=f\pi_2\Delta=f
\end{align*}
Now, let us suppose that $k'\colon X\to\T_2E$ satisfies $k'\pi_1=h$ and $k'\T_2q=f$. First, notice that:
\begin{align*}
&k'\pi_2p=k'\pi_1p=hp=g\\
&k'\pi_2\T q=k'\T_2q\pi_2=f\pi_2=f\pi_1
\end{align*}
Thus, also $k'\pi_2=h$. Therefore:
\begin{align*}
&k'\Delta=\<k'\pi_1,k'\pi_2\>=\<h,h\>=h\Delta=k
\end{align*}
Similarly, one can prove that such a pullback is preserved by all iterates of $\T$. Thus, the outer square of Equation~\eqref{equation:sum-naturality} is a tangent pullback and thus, by Lemma~\ref{lemma:pullback-lemma}, so is the left square.
\end{proof}
\end{lemma}

\begin{lemma}
\label{lemma:structural-natural-transformations-are-M}
The structural natural transformations of an $\M$-tangent category are $\M$-natural transformations.
\begin{proof}
By assumption, a monic $m$ of a tangent display system is an \'etale map. By Lemma~\ref{lemma:etale-maps}, the naturality squares of the structural natural transformations of $\TT$ with $m$ are tangent pullbacks.
\end{proof}
\end{lemma}

Since each tangent display system of monics is a display system of monics, given an $\M$-tangent category $(\X,\TT;\M)$,  $\Par(\X,\M)$ defines a split restriction category.

\begin{lemma}
\label{lemma:split-restriction-tangent-cats}
For an $\M$-tangent category $(\X,\TT;\M)$, the split restriction category $\Par(\X,\M)$ comes equipped with a restriction tangent structure.
\begin{proof}
To equip $\Par(\X,\M)$ with a restriction tangent structure, start by noticing that the $2$-functor $\Par$ sends the tangent bundle functor $\T\colon\X\to\X$ to a restriction functor and each $\M$-natural transformation to a total natural transformation. By Lemma~\ref{lemma:structural-natural-transformations-are-M}, all structural natural transformations of $\TT$ are $\M$-natural transformations. Thus, $\Par$ sends them to total natural transformations.
\par To prove that the resulting structure $\Par(\X,\M)[\TT]$ forms a restriction tangent structure on $\Par(\X,\M)$, first notice that all equational axioms only involve the structural natural transformations and thus are a direct consequence of the functoriality of $\Par$.
\par So, we just need to prove the axioms which involve the existence of restriction pullback diagrams, i.e., the existence of the restriction $n$-fold pullback of the projection along itself, and the universality of the vertical lift. In~\cite{cockett:restriction-categories-III-colims}, Cockett and Lack showed that in a split restriction category, restriction limits coincide with ordinary limits on total maps. By specializing their result for restriction pullbacks, we conclude that the restriction $n$-fold pullback of the projection along itself (recall that the projection is total) exists and the universal property of the vertical lift (which also involves only total maps) holds.
\end{proof}
\end{lemma}

Lemma~\ref{lemma:split-restriction-tangent-cats} only shows half of the story: the $2$-category of $\M$-tangent categories and split restriction tangent categories are indeed $2$-equivalent. Our proof of this statement is based on the formal theory of tangent objects which constitutes a formal approach to tangent category theory. Since this goes far beyond the scope of this paper, we decided to only state this result here and postpone the proof to a future paper.

\begin{theorem}
\label{theorem:equivalence-split-restriction-tangent-categories}
The assignment which sends an $\M$-tangent category $(\X,\TT;\M)$ to the split restriction tangent category $\Par(\X,\TT;\M)$ defined in Lemma~\ref{lemma:split-restriction-tangent-cats} extends to a $2$-equivalence between the $2$-category of $\M$-tangent categories and the 2-category of split restriction tangent categories.
\end{theorem}

In Definition~\ref{definition:M-tangent-category}, we assumed the monics of a tangent display system to be \'etale maps. This was a crucial assumption to prove that the structural natural transformations of an $\M$-tangent category are $\M$-natural transformations (Lemma~\ref{lemma:structural-natural-transformations-are-M}). We also mentioned in Remark~\ref{remark:tangent-display-system-monics} that such monics are automatically display \'etale maps. This suggests looking at tangent monic display \'etale maps. It turns out that such maps classify open subsets of smooth manifolds in differential geometry.
\par First, let us introduce the notion of a tangent monomorphism.

\begin{definition}
\label{definition:tangent-monic}
In a category $\X$ equipped with an endofunctor $\T$, a \textbf{$\T$-monic} is a monomorphism $m\colon A\to B$ of $\X$ such that, for each $n>0$, $\T^nm$ is still a monomorphism. When $\T$ represents the tangent bundle functor of a tangent category, we call a $\T$-monic a \textbf{tangent monic}.
\end{definition}

\begin{remark}
\label{remark:tangent-monics}
Since a tangent display system of monics is a tangent display system, every monic of such a system is automatically a tangent monic.
\end{remark}

\begin{lemma}
\label{lemma:open-subsets}
In the tangent category $\Smooth$ of (finite-dimensional) smooth manifolds, the inclusion $A\hookrightarrow M$ of an open subset $A$ into a manifold $M$ is a tangent monic display \'etale map. Furthermore, if $A\to M$ is a tangent monic display \'etale map between two smooth manifolds, its image is an open subset of $M$.
\begin{proof}
By standard arguments of differential geometry (see for instance~\cite[Corollary~2.3]{kolavr:differential-geometry}), every submersion is an open map. Thus, since \'etale maps are submersions, the image of each \'etale map is an open subset. Conversely, suppose that $m\colon A\hookrightarrow M$ is the inclusion of an open subset $A$ of $M$. For each point $x$ of $A$, there exists an open neighbourhood $U$ of $x$ of $M$ which is entirely included within $A$. Furthermore, since the topology of $M$ is obtained by gluing open balls of $\R^n$, without loss of generality, we can assume the neighbourhood $U$ to be diffeomorphic to an open ball of $\R^n$ centred at $x$, where $n$ is the dimension of $M$ at $x$.
\par Each tangent vector $v$ at $x$ of $M$ is the first derivative of a smooth path $\gamma\colon\R\to U$ such that $\gamma(0)=x$. Up to a diffeomorphism, we can assume each smooth path $\gamma$ to represent a ray in one of the directions of the open ball $U\subseteq\R^n$ from the centre to a point of the boundary.
\par Since $U\subseteq A$, $\gamma$ also represents a tangent vector at $x$ of $A$ and viceversa. In particular, the differential $\d m_x$ of the inclusion $m\colon A\hookrightarrow M$ at $x$ is bijective. Thus, $m$ is \'etale. Moreover, each $\T^nm$ is monic since it is injective and finally, \'etale maps in the category of smooth manifolds are automatically display \'etale (see Proposition~\ref{proposition:submersions-etale-differential-geometry}).
\end{proof}
\end{lemma}

Thanks to this, we can introduce the notion of an open subset in a tangent category.

\begin{definition}
\label{definition:open-subset}
An \textbf{open subobject} of an object $M$ in a tangent category $(\X,\TT)$ consists of an object $A$ together with a tangent monic display \'etale map $m\colon A\hookrightarrow M$.
\end{definition}

In Lemma~\ref{lemma:open-subsets}, to prove that the image of a tangent monic display \'etale map is open we used that each submersion is always an open function. This suggests one could consider tangent monic display submersions in Definition~\ref{definition:open-subset} instead of display \'etale maps. However, the next lemma tells us that these two concepts coincide.

\begin{lemma}
\label{lemma:tangent-monics}
A tangent monic submersion $q\colon E\to M$ in a tangent category $(\X,\TT)$ is \'etale.
\begin{proof}
Let us consider two morphisms $f\colon X\to E$ and $g\colon X\to\T M$ satisfying $fq=gp$. Since $q$ is a submersion, there exists a morphism $h\colon X\to\T E$ satisfying $hp=f$ and $h\T q=g$. Suppose that $h'\colon X\to\T E$ satisfies the same equations. In particular, $h'\T q=h\T q$. Since $q$ is tangent monic, $\T q$ is monic, thus $h'=h$.
\end{proof}
\end{lemma}

The category of classes of isomorphisms of open subobjects of a tangent category $(\X,\TT)$ forms a poset $\Open(\X,\TT)$ which admits meets defined by pullbacks. In general, $\Open(\X,\TT)$ does not admit joins. When it does, $\Open(\X,\TT)$ defines a frame for $(\X,\TT)$. Every morphism $f\colon M\to N$ of a tangent category is continuous with respect to this frame. Concretely, this means that the pullback of each open subobject $m\colon B\hookrightarrow N$ along $f$ exists and is an open subobject of $M$. Furthermore, submersions are always open maps.

\begin{lemma}
\label{lemma:submersions-are-open}
Every display submersion $q\colon E\to M$ of a tangent category $(\X,\TT)$ is open, meaning that the composition:
\begin{align*}
&A\xrightarrow{m}E\xrightarrow{q}M
\end{align*}
is an open subobject of $M$, for each open subobject $m\colon A\to E$.
\begin{proof}
Since $m$ is tangent monic, $mq$ is also tangent monic. Furthermore, since display submersions are closed under composition, $mq$ is also a display submersion. By Lemma~\ref{lemma:tangent-monics}, every tangent monic submersion is \'etale.
\end{proof}
\end{lemma}

One can employ the class of open subobjects to construct a split restriction tangent category out of each tangent category.

\begin{theorem}
\label{theorem:open-restriction}
The class $\Open(\X,\TT)$ of monics underlying open subobjects of a tangent category $(\X,\TT)$ is the maximal tangent display system of monics of $(\X,\TT)$. In particular, every tangent category is canonically embedded in a split restriction tangent category $\Par(\X,\TT)\=\Par(\X,\TT;\Open(\X,\TT))$.
\begin{proof}
We already know that tangent display maps form the maximal tangent display system. Furthermore, tangent monics and display tangent \'etale maps are closed under composition and each isomorphism is tangent monic, tangent display (since pullbacks along isomorphisms always exist), and \'etale.
\end{proof}
\end{theorem}


\section{Conclusions}
\label{section:conclusions}
In this paper, we have argued that tangent display maps are an extremely important class of maps in an arbitrary tangent category. As shown by the results of this paper, not only do they recreate important classes of maps in specific examples (e.g., tangent display maps in the category of smooth manifolds are precisely the submersions -  Theorem~\ref{theorem:classification-submersions}), but one can also derive important results related to them (Theorems~\ref{theorem:T-display-maps-form-maximal-T-display-system},~\ref{theorem:-retractive-tangent-display-systems},~\ref{theorem:cauchy-completion-tangent-cats},~\ref{theorem:fully-display-tangent-cats},~\ref{theorem:adjunction-Term-Slice}, and Theorem~\ref{theorem:open-restriction}). Given this, we suggest that whenever one needs to consider maps which require tangent pullbacks in a tangent category, rather than making a choice of a system of such maps, it is preferable to simply ask that such maps be tangent display maps.

As another example, this work was inspired by thinking about the issue of pullbacks when considering more general types of connections in a tangent category. Previously, work on connections in tangent categories (\cite{cockett:connections}) has focused on connections on differential bundles, which abstract connections on vector bundles. However, differential geometry considers connections not just on vector bundles, but connections on fibre bundles, or even more generally (going back to the work of Ehresmann \cite{ehresmann:connections}) connections on submersions. A natural question is then: on which maps should one define connections on in an arbitrary tangent category? 

Since submersions are the most general type of map differential geometry considers connections on, it is natural to assume that that is the collection one should consider, using the abstract definition (see Definition \ref{definition:submersions}). However, one is immediately met with a problem when trying to define a connection on such a map $q: E \to M$: one needs the ``horizontal space'' of $q$, which is the pullback of $q$ along $p_M$:
\[
\begin{tikzcd}
\T M \times_M E & E \\
\T M & M
\arrow[, from=1-1, to=1-2]
\arrow[from=1-1, to=2-1]
\arrow["{p_M}"', from=2-1, to=2-2]
\arrow["q", from=1-2, to=2-2]
\end{tikzcd}
\]
There is no reason why a submersion in an arbitrary tangent category need admit such a pullback. The problem gets worse when one considers how to define the curvature of such connections, where additional pullbacks are needed (and typically one needs these pullbacks to be preserved by $\T$). 

One possibility is to instead consider some system of maps which are closed under pullbacks and applications of $\T$. But as discussed previously, this is an artificial choice and level of generality which does not seem to bring any benefit.

However, the results of this paper tell us what we should do instead: instead of considering connections on an arbitrary submersion or being forced to add extra structure in the form of a system of maps, one can instead simply consider connections on an arbitrary tangent display map. One immediately then has all the pullbacks one needs. And, as an added bonus, by Theorem~\ref{theorem:classification-submersions}, when applied to the tangent category of smooth manifolds, these maps are exactly the most general types of maps one considers connections on in differential geometry, namely, the submersions.  Thus, one follow-up to this paper will be defining and working with the notion of a connection on a tangent display map in a tangent category. 

Another avenue to pursue is the canonical split restriction structure given by the open subobjects of a tangent category (Section \ref{subsection:restriction}).  For example, can one determine under what conditions this restriction structure has joins?  

Thus, we hope to build upon the work of this paper in several different ways.  However, we also hope that that the results of this paper will encourage others to use tangent display maps in their own work on tangent categories.



\begingroup
\begin{thebibliography}{20}

\bibitem{borceux:cauchy-completion}Borceux, F. \& Dejean, D. Cauchy completion in category theory. {\em Cahiers Topologie Géom. Différentielle Catég.}. \textbf{27}, 133-146 (1986)
\bibitem{cockett:tangent-cats}Cockett, J. \& Cruttwell, G. Differential Structure, Tangent Structure, and SDG. {\em Applied Categorical Structures}. \textbf{22}, 331-417 (2014)
\bibitem{cockett:connections}Cockett, J. \& Cruttwell, G. Connections in tangent categories. {\em Theory And Applications Of Categories}. \textbf{32}, 835-888 (2017)
\bibitem{cockett:differential-bundles}Cockett, J. \& Cruttwell, G. Differential Bundles and Fibrations for Tangent Categories. {\em Cahiers De Topologie Et Géométrie Différentielle Catégoriques}. \textbf{LIX} pp. 10-92 (2018)
\bibitem{cockett:differential-equations}Cockett, J., Cruttwell, G. \& Lemay, J. Differential equations in a tangent category i: Complete vector fields, flows, and exponentials. {\em Applied Categorical Structures}. pp. 1-53 (2021)
\bibitem{cockett:restriction-categories-I}Cockett, J. \& Lack, S. Restriction categories. I. Categories of partial maps. {\em Theoret. Comput. Sci.}. \textbf{270}, 223-259 (2002), https://doi.org/10.1016/S0304-3975(00)00382-0
\bibitem{cockett:restriction-categories-III-colims}Cockett, J. \& Lack, S. Restriction categories. III. Colimits, partial limits and extensivity. {\em Math. Structures Comput. Sci.}. \textbf{17}, 775-817 (2007), https://doi.org/10.1017/S0960129507006056
\bibitem{cruttwell:algebraic-geometry}Cruttwell, G. \& Lemay, J. Differential bundles in commutative algebra and algebraic geometry. {\em Theory Appl. Categ.}. \textbf{39} pp. Paper No. 36, 1077-1120 (2023)
\bibitem{cruttwell:reverse-tangent-cats}Cruttwell, G. \& Lemay, J. Reverse tangent categories. {\em 32nd EACSL Annual Conference On Computer Science Logic}. \textbf{288} pp. Art. No. 21, 21 (2024), https://doi.org/10.4230/lipics.csl.2024.21
\bibitem{ehresmann:connections}Erhesmann, C. Les connexions infinitésimales dans un espace fibré différentiable. {\em Colloque De Topologie}. pp. 29-55 (1950)
\bibitem{garner:embedding-theorem-tangent-cats}Garner, R. An embedding theorem for tangent categories. {\em Advances In Mathematics}. \textbf{323} pp. 668-687 (2018)
\bibitem{kolavr:differential-geometry}Kolář, I., Michor, P. \& Slovák, J. Natural operations in differential geometry. (Springer-Verlag, Berlin,1993), https://doi.org/10.1007/978-3-662-02950-3
\bibitem{lanfranchi:grothendieck-tangent-cats}Lanfranchi, M. The Grothendieck construction in the context of tangent categories.  (2023)
\bibitem{lanfranchi:differential-bundles-operadic-affine-schemes}Lanfranchi, M. The differential bundles of the geometric tangent category of an operad. {\em Appl. Categ. Structures}. \textbf{32}, Paper No. 25, 43 (2024), https://doi.org/10.1007/s10485-024-09771-2
\bibitem{lawvere:metric-spaces-cauchy-completeness}Lawvere, F. Metric spaces, generalized logic, and closed categories. {\em Repr. Theory Appl. Categ.}., 1-37 (2002)
\bibitem{leung:weil-algebras}Leung, P. Classifying tangent structures using Weil algebras. {\em Theory And Applications Of Categories}. \textbf{32} pp. 286-337 (2017)
\bibitem{macadam:vector-bundles}MacAdam, B. Vector bundles and differential bundles in the category of smooth manifolds. {\em Appl. Categ. Structures}. \textbf{29}, 285-310 (2021), https://doi.org/10.1007/s10485-020-09617-7
\bibitem{metzler:stacks}Metzler, D. Topological and Smooth Stacks. {\em ArXiv:0306176}. (2003)
\bibitem{rosicky:tangent-cats}Rosický, J. Abstract Tangent Functors. {\em Diagrammes}. \textbf{12} pp. JR1-JR11 (1984)

\end{thebibliography}
\endgroup
\end{document}